\documentclass[smallextended]{svjour3}
\smartqed  % flush right qed marks, e.g. at end of proof
\usepackage{graphicx}
\usepackage{graphicx}
\usepackage{amsmath}
\usepackage{amsfonts}
\usepackage{amssymb}
\usepackage{mathrsfs}

\usepackage{caption}
\usepackage{subcaption} % for multiple floats

\numberwithin{equation}{section}
\numberwithin{figure}{section}
\numberwithin{table}{section}
\numberwithin{theorem}{section}
\numberwithin{lemma}{section}
\numberwithin{proposition}{section}
\numberwithin{corollary}{section}
\numberwithin{definition}{section}
\numberwithin{remark}{section}

\usepackage{color}
\definecolor{newblue}{RGB}{94,89,144}
\definecolor{newblue2}{cmyk}{1,0.6,0,0.06}
\usepackage[pdfborder={0 0 0},
  colorlinks=true,
  citecolor=newblue,
  linkcolor=newblue2,
  urlcolor=newblue]{hyperref}

\usepackage{xifthen}

\newcommand{\myref}[2][]{%
  \ifthenelse{\isempty{#1}}%
  {\ref{#2}}% if #1 is empty
  {\hyperref[#2]{#1~\ref*{#2}}}% if #1 is not empty
}

\newcommand{\mycite}[2]{\textnormal{\cite[#1]{#2}}}

\newcommand{\dd}{\,\mathrm{d}}
\newcommand{\abs}[1]{\left|#1\right|}
\newcommand{\norm}[1]{\left\|#1\right\|}
\newcommand{\R}{\mathbb{R}}
\newcommand{\N}{\mathcal{N}}
\newcommand{\A}{\mathcal{A}}
\newcommand{\B}{\mathcal{B}}
\newcommand{\E}{\mathcal{E}}
\newcommand{\M}{\mathcal{M}}
\renewcommand{\S}{\mathcal{S}}
\newcommand{\G}{\mathcal{G}}

\newcommand{\T}{\mathscr{T}}

\renewcommand{\H}{\mathcal{H}}

\renewcommand{\Psi}{\mathcal{F}}

\DeclareMathOperator*{\esssup}{ess\, sup}
\DeclareMathOperator{\diam}{diam}
\DeclareMathOperator{\trace}{trace}

\newcommand{\md}{\partial^\bullet}
\newcommand{\nd}{\partial^\circ}
\newcommand{\Rh}{\mathcal{R}_h}

\newcommand{\eps}{\varepsilon}
\newcommand{\weakto}{\rightharpoonup}
\newcommand{\id}{\mathrm{Id}}

\renewcommand{\tilde}{\widetilde}

\newcommand{\dt}{\frac{\mathrm{d}}{\mathrm{d}t}}
\newcommand{\ds}{\frac{\mathrm{d}}{\mathrm{d}s}}

\journalname{Numerische Mathematik}
\begin{document}

\title{Evolving surface finite element method for the Cahn-Hilliard equation%
  \thanks{%
    The work of C.\ M.\ Elliott was supported by the UK Engineering and Physical Sciences Research Council EPSRC Grant EP/G010404 and
    the work of T.\ Ranner was supported by a EPSRC Ph.D. studentship (Grant EP/P504333/1 and EP/P50516X/1) and the Warwick Impact Fund.
  }
}
%\subtitle{Do you have a subtitle?\\ If so, write it here}

%\titlerunning{Short form of title}        % if too long for running head

\author{Charles M. Elliott         \and
        Thomas Ranner %etc.
}

%\authorrunning{Short form of author list} % if too long for running head

\institute{C.M.Elliott \at
              Mathematics Institute, Zeeman Building, University of Warwick, Coventry. CV4 7AL. UK \\
              Tel.: +44 (0)24 7615 0773\\
%              Fax: +123-45-678910\\
              \email{C.M.Elliott@warwick.ac.uk}           %  \\
%             \emph{Present address:} of F. Author  %  if needed
           \and
           T. Ranner \at
              Mathematics Institute, Zeeman Building, University of Warwick, Coventry. CV4 7AL. UK \\
%              Tel.: +44 (0)24 7615 0773\\
%              Fax: +123-45-678910\\
              \emph{Present address:} %
              School of Computing, University of Leeds. LS2 9JT. UK \\
              E-mail: T.Ranner@leeds.ac.uk
            }

\date{The final publication is available at springerlink.com. DOI: 10.1007/s00211-014-0644-y}
% The correct dates will be entered by the editor

\maketitle

\begin{abstract}
We use the evolving surface finite element method to solve a Cahn-Hilliard equation on an evolving surface with prescribed velocity. We start by deriving the equation using a conservation law and appropriate transport formulae and provide the necessary functional analytic setting. The finite element method relies on evolving an initial triangulation by moving the nodes according to the prescribed velocity. We go on to show a rigorous well-posedness result for the continuous equations by showing convergence, along a subsequence, of the finite element scheme. We conclude the paper by deriving error estimates and present various numerical examples.
% \PACS{PACS code1 \and PACS code2 \and more}
% \subclass{MSC code1 \and MSC code2 \and more}
\end{abstract}

%--- introduction -----------------------------------------------------------%           
                                                                                         
\section{Introduction}
\label{sec:introduction-1}

In this paper, we will study a Cahn-Hilliard equation posed on an evolving surface with prescribed velocity. The key methodology is to discretise the equations using the evolving surface finite element method \cite{DziEll07} originally proposed for a surface heat equation. The idea is to take a triangulation of the initial surface and evolve the nodes along the velocity field. This leads to a family of discrete surfaces on which we can pose a variational form of the Cahn-Hilliard equation.

There are two key results in this paper: first, we show well posedness of the continuous scheme and, second, we show convergence of a finite element scheme. The well posedness result is proven by rigorously showing convergence, along a subsequence, of the discrete scheme. In contrast to the planar setting, there are extra difficulties in this work since the classical Bochner space set-up is unavailable to us. The finite element method is analysed under the assumption of higher regularity of the solution and shown to converge to the true solution quadratically with respect to the mesh size in an $L^2$ norm. The paper concludes with some numerical examples to show various properties of the methodology.

%--- equations --------------------------------------------------------------%

\subsection{The Cahn-Hilliard equation}

We assume we are given an evolving surface $\{ \Gamma(t) \}$, for $t \in [0,T]$, which evolves according to a given underlying velocity field $v$ which can be decomposed into normal $(v_\nu)$ and tangential components $(v_\tau)$ so that $v = v_\nu + v_\tau$. We seek a solution $u$ of
\begin{equation}
  \label{eq:111}
  \md u + u \nabla_\Gamma \cdot v = \Delta_\Gamma \left( -\eps \Delta_\Gamma u + \frac1\eps \psi'(u) \right) \quad \mbox{ on } \bigcup_{t \in (0,T)} \Gamma(t) \times \{ t \}
\end{equation}
subject to the initial condition
\begin{equation}
  \label{eq:112}
  u( \cdot, 0 ) = u_0 \quad \mbox{ on } \Gamma(0) = \Gamma_0.
\end{equation}
Here $\md u$ denotes the material derivative of $u$ and $\Delta_\Gamma u$ the Laplace-Beltrami operator of $u$. The function $\psi$ is a double well potential, which we will take to be given by
\begin{equation}
  \label{eq:113}
  \psi( z ) = \frac14 ( z^2 - 1 )^2.
\end{equation}

The behaviour of the Cahn-Hilliard equation in the planar case is well studied \cite{Ell89}. Extra effects such as spatial or concentration dependent mobilities or more physically realistic potentials could also be solved with similar methods to those suggested in this paper. Such considerations are left for future work.

This Cahn-Hilliard equation is a simplification of the model for surface dissolution set out in \cite{EilEll08,ErlAziKar01} arising from a conservation law. The model \cite{MerPtaKuh12} takes a different approach and considers a gradient flow for an energy consisting of the sum of the Ginzburg-Landau functional and a Helfrich energy on a stationary surface. One could alternatively couple the evolution of the surface to the surface field $u$ and recover a gradient flow of the Ginzburg-Landau functional \cite{EllSti10b,EllSti10a}.

The results in this work can be seen as a generalisation of the work of \cite{DuJuTia09} to evolving surfaces. That work considers a fully discrete approximation of a Cahn-Hilliard equation posed on a two-dimensional stationary surface with boundary (with a zero Dirichlet boundary condition) under the assumption $u_0 \in H^1_0(\Gamma) \cap H^2(\Gamma)$ and $\Delta_\Gamma u_0 \in H^1_0(\Gamma) \cap W^{1,2+\gamma}(\Gamma)$ for $\gamma \in (0,1)$. Their method uses a triangulated surface for the spatial discretisation and a Crank-Nicolson scheme in time. They show an error estimate of the form
\begin{equation*}
  \max_m \norm{ u_h^m - u^{-\ell}( t_m ) }_{L^2(\Gamma_h)} \le c ( h^2 + \tau^2 ),
\end{equation*}
where $0 = t_0 < t_1 < \ldots < t_m < \ldots < t_M = T$ is a partition of time with fixed time step $\tau$ and $u^{-\ell}$ is the inverse lift \eqref{eq:76} of the continuous solution $u$. 

%--- outline of paper -------------------------------------------------------%

\subsection{Outline of paper}

The paper is laid out as follows. In Section two, we will derive a Cahn-Hilliard equation on an evolving surface using a local conservation law. We introduce the notation for partial differential equations on evolving surfaces taken from \cite{DecDziEll05,DziEll13a} and state any assumptions on the smoothness of the surfaces and its evolution we require. The third section introduces a finite element discretisation of the continuous equations. We describe the process of triangulating an evolving surface and how we formulate the space discrete-time continuous problem as a system of ordinary differential equations. This section is completed by showing some domain perturbation results relating geometric quantities on the discrete and smooth surfaces. Well posedness of the continuous equations is addressed in the fourth section. An existence result is achieved by showing convergence, along a subsequence, of the discrete solutions as the mesh size tends to zero. In Section five, we analyse the errors introduced by our finite element scheme and go on to show an optimal order error estimate. Some numerical experiments are shown in the sixth section backing up the analytical results.

We will use a Gronwall inequality as a standard tool in the analysis which leads to exponential dependence on $\eps$ in most bounds. We are not interested in taking $\eps \to 0$ in this work so will simply write $c_\eps$ for a generic constant which depends on $\eps$.

%--- derivation of continuous equations -------------------------------------%           
                                                                                         
\section{Derivation of continuous equations}

In this section, we will derive a Cahn-Hilliard equation on an evolving surface as a conservative advection-diffusion equation. We will also introduce functional analytic setting and definition of solution that will be used.

\subsection{Assumptions on the evolving surface}
\label{sec:ch-notation}

Given a final time $T>0$, for each time $t \in [0,T]$, we write $\Gamma(t)$ for a compact, smooth, connected $n$-dimensional hypersurface in $\R^{n+1}$ for $n = 1,2$ or $3$ and $\Gamma_0 = \Gamma(0)$. We assume that $\Gamma(t)$ is the boundary of an open, bounded domain $\Omega(t)$. It follows that $\Gamma(t)$ admits a description as the zero level set of a signed distance function $d(\cdot, t) \colon \R^{n+1} \to \R$ so that $d(\cdot, t) < 0$ in $\Omega(t)$ and $d(\cdot, t) > 0$ in $\bar{\Omega}(t)^c$. We denote by $\G_T$ for the space-time domain given by
\begin{equation}
  \label{eq:99}
  \G_T = \bigcup_{t \in [0,T]} \Gamma(t) \times \{ t \}.
\end{equation}

\begin{figure}
  \centering
  \includegraphics{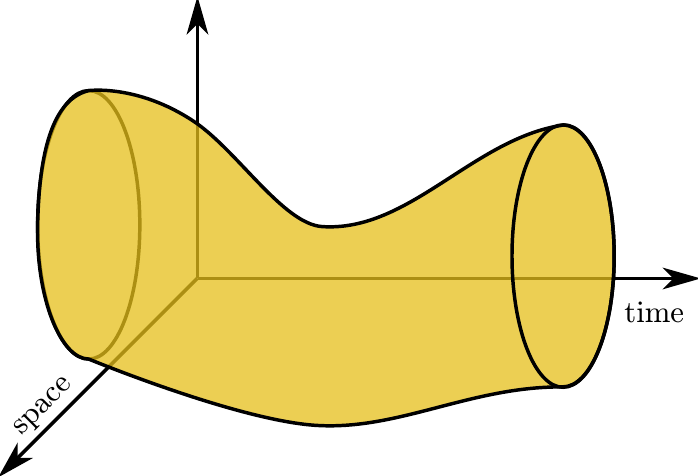}
  \caption[Sketch of space-time domain $\G_T$]{A sketch of the space-time domain $\G_T$.}
\end{figure}

For our analysis, it is sufficient to consider $d(\cdot, t)$ locally to $\Gamma(t)$. We restrict our considerations to $\N(t)$, an open neighbourhood of $\Gamma(t)$. We choose $\N(t)$ so that $\abs{ \nabla d(x,t) } \neq 0$ for $x \in \N(t)$ and assume that
\begin{equation*}
  d, d_t, d_{x_i}, d_{x_i x_j} \in C^2(\N_T) \quad \mbox{ for } i,j=1,\ldots,n+1;
\end{equation*}
here $\N_T = \bigcup_{t \in [0,T]} \N(t) \times \{ t \}$. The orientation of $\Gamma(t)$ is fixed by choosing $\nu$ as the outward pointing normal, so that $\nu(x,t) = \nabla d(x,t)$. For $(x,t) \in \G_T$, we denote $P = P(x,t)$ the projection operator onto the tangent space $T_x \Gamma(t)$, given by $P_{ij}(x,t) = \delta_{ij} - \nu_i(x,t) \nu_j(x,t)$ and by $\H = \H(x,t)$ the (extended) Weingarten map (or shape operator),
\begin{equation*}
  \H_{ij}(x,t) = (\nu_i(x,t) )_{x_j} = d_{x_i x_j}(x,t).
\end{equation*}
We will use the fact that $P \H = \H P = \H$. Finally, we denote by $H = H(x,t)$ the mean curvature of $\Gamma(t)$
\begin{equation*}
  H(x,t) = \trace \H(x,t) = \sum_{i=1}^{n+1} \H_{ii}(x,t).
\end{equation*}

For a function $\eta \colon \Gamma(t) \to \R$, we define its tangential gradient $\nabla_\Gamma \eta$ by
\begin{equation*}
  \nabla_\Gamma \eta = \nabla \tilde\eta - \nabla \tilde\eta \cdot \nu \nu = P \nabla \tilde\eta,
\end{equation*}
where $\tilde\eta$ is a smooth extension of $\eta$ away from $\Gamma(t)$. It can be shown that this definition is independent of the choice of extension. We denote the $n+1$ components of $\nabla_\Gamma \eta$ by
\begin{equation*}
  \nabla_\Gamma \eta = ( \underline{D}_1 \eta, \ldots, \underline{D}_{n+1} \eta ).
\end{equation*}
The Laplace-Beltrami operator is given by
\begin{equation*}
  \Delta_\Gamma \eta = \nabla_\Gamma \cdot \nabla_\Gamma \eta = \sum_{j=1}^{n+1} \underline{D}_j \underline{D}_j \eta.
\end{equation*}

We will denote by $\mathrm{d} \sigma$ the surface measure on $\Gamma(t)$ which admits the following formula for partial integration for a portion $\M(t) \subseteq \Gamma(t)$ \cite[Theorem~2.10]{DziEll13a}:
\begin{equation}
  \label{eq:int-by-parts}
  \int_{\M(t)} \nabla_\Gamma \eta \dd \sigma = \int_{\M(t)} \eta H \nu \dd \sigma + \int_{\partial \M(t)} \eta \mu \dd \sigma,
\end{equation}
where $\mu$ is the co-normal to $\partial \M(t)$ which is normal to $\partial \M(t)$ but tangent to $\Gamma(t)$. If $\M(t) = \Gamma(t)$ and has no boundary, the boundary term vanishes. Furthermore, we have a Green's formula on $\Gamma(t)$ \cite[Theorem~2.14]{DziEll13a}:
\begin{equation}
  \label{eq:greens-form}
  \int_{\Gamma(t)} \nabla_\Gamma \eta \cdot \nabla_\Gamma \varphi \dd \sigma
  = - \int_{\Gamma(t)} \varphi \Delta_\Gamma \eta \dd \sigma.
\end{equation}

These formulae allow the definition of weak derivatives and Sobolev spaces. We define the space $W^{1,q}(\Gamma(t))$ by
\begin{equation*}
  W^{1,q}(\Gamma(t)) := \left\{ \eta \in L^q(\Gamma(t)) : \underline{D}_j \eta \in L^q(\Gamma(t)) \mbox{ for } j = 1,\ldots,n+1 \right\},
\end{equation*}
with norm
\begin{equation*}
  \norm{ \eta }_{W^{1,q}(\Gamma(t))} = \left( \norm{ \eta }_{L^q(\Gamma(t))}^q + \norm{ \nabla_\Gamma \eta }_{L^q(\Gamma(t))}^q \right)^{\frac1q}.
\end{equation*}
This can be easily extended to higher order spaces. See \cite{DziEll13a} for details. We will use the notation $H^k(\Gamma(t))$ for $W^{k,2}(\Gamma(t))$.

We will make use of the following Sobolev embeddings:
\begin{lemma}
  [\mycite{Theorems~2.5 and 2.6}{Heb00}]
  \label{lem:sobolev-embed}
  For $\Gamma(t)$ as above, we have
  \begin{align}
    W^{1,q}(\Gamma(t)) \subset
    \begin{cases} 
      L^{ nq / ( n-q )}( \Gamma(t) ) & \mbox{ for }  q < n \\
      C^0( \Gamma(t) ) & \mbox{ for } q > n.
    \end{cases}
  \end{align}
  Furthermore there exists a constant $c = c(n,q)$, independent of $t$, such that for any $\eta \in W^{1,q}(\Gamma(t))$,
  \begin{subequations}
    \begin{align}
      \label{eq:128}
      \norm{ \eta }_{L^{n q / (n-q)}(\Gamma(t))} & \le c \norm{ \eta }_{W^{1,q}(\Gamma(t))} && \mbox{ for } q < n \\
      \norm{ \eta }_{L^\infty(\Gamma(t))} & \le c \norm{ \eta }_{W^{1,q}(\Gamma(t))} && \mbox{ for } q > n.
    \end{align}
  \end{subequations}
\end{lemma}

In particular, this allows us to embed $H^1(\Gamma(t))$ in $L^6(\Gamma(t))$ for all dimensions $(n=1,2,3)$ so that $\norm{ \psi'(\eta) }_{L^2(\Gamma(t))} \le c \big( \norm{ \eta }_{H^1(\Gamma(t))}^3 + \norm{ \eta }_{H^1(\Gamma(t))} \big)$.

Further, we assume that for each $(x,t) \in \N_T$ there exists a unique $p = p(x,t) \in \Gamma(t)$, such that
\begin{equation}
\label{eq:101}
  x = p(x,t) + d(x,t) \nu( p(x,t), t).
\end{equation}
See \cite[Chapter~14]{GilTru01} for a proof. We extend $\nu, P$ and $\H$ to functions on $\N_T$ by setting
\begin{equation*}
  \nu(x,t) = \nu(p(x,t),t) = \nabla d(x,t),
\end{equation*}
and similarly $P(x,t) = P(p(x,t),t) = \id - \nu(x,t) \otimes \nu(x,t)$ and $\H(x,t) = \nabla^2 d(x,t)$ for $(x,t) \in \N_T$.

Although it is sufficient to describe the evolution of the surface through a normal velocity, we wish to consider material surfaces for which a material particle, at $X(t)$ on $\Gamma(t)$, has a material velocity $\dot{X}(t)$ not necessarily only in the normal direction. The normal velocity of the surface can be calculated to be $v_\nu = -d_t \nu$. We say $v_\tau$ is a tangential velocity field if $v_\tau \cdot \nu = 0$ in $\N_T$. Given a tangential velocity field $v_\tau$, we call
\begin{equation*}
  v := v_\tau + v_\nu
\end{equation*}
a material velocity field. We assume that we are given a global velocity field $v$ so that points $X(t)$ evolve with the velocity $\dot{X}(t) = v( X(t), t )$. We will assume that $v \in C^2(\N_T)$.

\subsection{Material derivative and transport formulae}

Given a family of surfaces $\{ \Gamma(t) \}$ evolving in time with normal velocity field $v_\nu$, we define the normal time derivative $\nd$ of a function $\eta\colon \G_T \to \R$ by
\begin{equation}
  \label{eq:100}
  \nd \eta := \frac{\partial \tilde\eta}{\partial t} + v_\nu \cdot \nabla \tilde\eta.
\end{equation}
Here, $\tilde\eta$ denotes a smooth extension of $\eta$ to $\N_T$. This derivative describes how a quantity $\eta$ evolves in time with respect to the evolution of $\Gamma(t)$. It can be shown that this definition is an intrinsic surface derivative, independent of the choice of extension.

Given a tangential vector field $v_\tau$, we define the material derivative of a scalar function $\eta \colon \G_T \to \R$, by
\begin{equation*}
  \md \eta := \nd \eta + v_\tau \cdot \nabla_\Gamma \eta = \frac{\partial \tilde\eta}{\partial t} + v \cdot \nabla \tilde\eta.
\end{equation*}

The following formula shows the significance of the material derivative. The result is a generalisation of the classical Reynolds' Transport Formula to curved domains.

\begin{lemma}
  [Transport formula \mycite{Lemma 2.1}{DziEll13}]
  \label{lem:transport}
  Let $\M(t)$ be an evolving surface with normal velocity $v_\nu$. Let $v_\tau$ be a tangential velocity field on $\M(t)$. Let the boundary $\partial \M(t)$ evolve with velocity $v = v_\nu + v_\tau$. Assume that $\eta, \varphi$ are functions such that all the following quantities exist. Then, we obtain the identity
  \begin{equation}
    \label{eq:transport}
    \dt \int_{\M(t)} \eta \dd \sigma = \int_{\M(t)} \md \eta + \eta \nabla_\Gamma \cdot v \dd \sigma.
  \end{equation}

  Furthermore, we have
  \begin{equation}
    \label{eq:96}
    \dt \int_{\M(t)} \eta \varphi \dd \sigma
    = \int_{\M(t)} \md \eta \, \varphi + \eta \, \md \varphi + \eta \varphi \, \nabla_\Gamma \cdot v \dd \sigma.
  \end{equation}
  Let $\A = \A(x,t)$ be a matrix which is positive definite on the tangent space to $\Gamma(t)$. Denote by $D(v)$ the rate of deformation tensor given by
  \begin{equation}
    D(v)_{ij} = \frac12 \sum_{k=1}^{n+1} \big( \A_{ik} \underline{D}_k v_j + \A_{jk} \underline{D}_k v_i \big) \quad \mbox{ for } i,j = 1,\ldots, n+1,
  \end{equation}
  and by $\B(v)$ the tensor
  \begin{equation}
    \label{eq:98}
    \B(v) := \md \A + \nabla_\Gamma \cdot v \A - 2 D(v).
  \end{equation}
  Then we have the formula
  \begin{equation}
    \label{eq:103}
    \begin{aligned}
    \dt \int_{\M(t)} \A \nabla_\Gamma \eta \cdot \nabla_\Gamma \varphi \dd \sigma
    & = \int_{\M(t)} \A \nabla_\Gamma \md \eta \cdot \nabla_\Gamma \varphi
    + \A \nabla_\Gamma \eta \cdot \nabla_\Gamma \md \varphi \dd \sigma \\
    & \qquad+ \int_{\M(t)} \B(v) \nabla_\Gamma \eta \cdot \nabla_\Gamma \varphi \dd \sigma.
  \end{aligned}
\end{equation}
\end{lemma}

We conclude this subsection with a result allowing us to extend functions defined on one surface to the whole space-time domain.

\begin{lemma}
  \label{lem:pull-back}
  Fix $t \in [0,T]$ and let $\eta \in H^1(\Gamma(t))$, respectively $C^1(\Gamma(t))$. Then there exists an extension $\tilde\eta \colon \G_T \to \R$ such that $\tilde\eta|_{t} = \eta$ and $\tilde\eta \in H^1(\Gamma(s))$, resp. $C^1(\Gamma(s))$, for all times $s \in [0,T]$ and $\md \tilde\eta = 0$.
\end{lemma}

\begin{proof}
  The ordinary differential equation:
  \begin{equation*}
    \ds X(s) = v( X(s), s ) \quad \mbox{ for } s \in [0,T], \qquad X(t) = x,
  \end{equation*}
  determines a flow $\phi_s(x)$ on $\G_T$ for $x \in \Gamma(t)$ such that
  \begin{equation*}
    \phi_s(x) \in \Gamma(s) \quad \mbox{ for all } s \in [0,T] 
    \quad \mbox{ and }
    \phi_t(x) = x.
  \end{equation*}
  Our assumptions on $v$ imply that $\phi_s \colon \Gamma(t) \to \Gamma(s)$ and $(\phi_s)^{-1}\colon \Gamma(s) \to \Gamma(t)$ are both $C^1$ mappings \cite[Theorem~3.1]{Har64}.

  We define the extension $\tilde\eta$ by
  \begin{equation*}
    \tilde\eta( x, s ) := \eta( (\phi_s)^{-1}(x) ) \quad \mbox{ for } (x,s) \in \G_T.
  \end{equation*}
  It is clear that since $(\phi_s)^{-1} \in C^1(\Gamma(t);\Gamma(s))$, we have $\tilde\eta \in H^1(\Gamma(s))$ (resp. $C^1(\Gamma(s))$) for all times $s \in [0,T]$.
  
  Finally, we can calculate for $y = (\phi_s)^{-1}(x)$,
  \begin{equation*}
    \md \tilde\eta(x,s) = \ds \tilde \eta ( \phi_s(y), s ) = \ds \eta( y ) = 0 \quad \mbox{ for } (x,s) \in \G_T,
  \end{equation*}
  which shows the result. \qed
\end{proof}

\subsection{Derivation of Cahn-Hilliard equations}

We will consider a conservation law on an evolving surface with a diffusive flux driven by a chemical potential. This is the approach taken by \cite{ErlAziKar01}. In general, the Ginzburg-Landau functional on $\Gamma(t)$ will not decrease along the trajectory of solutions.

Let $u$ represent a density of a scalar quantity on $\Gamma(t)$. Following \cite{DziEll13}, we arrive at the pointwise conservation law
\begin{equation}
  \label{eq:64}
  \nd u + u \nabla_\Gamma \cdot v_\nu + \nabla_\Gamma \cdot q = 0.
\end{equation}
Here $q$ represents the tangential flux of $u$ on $\{ \Gamma(t) \}$.

We will assume that the flux $q$ is the sum of a diffusive flux $q_d$ and an advective flux $q_a$:
\begin{equation*}
  q_d = - \nabla_\Gamma w \quad \mbox{ and } \quad q_a = u v_\tau.
\end{equation*}
The diffusive flux is driven by the gradient of chemical potential $w$ gives us the split system \cite{EllFreMil89}
\begin{subequations}
  \label{eq:ch-split}
  \begin{align}
    \label{eq:ch-w}
    \md u + u \nabla_\Gamma \cdot v - \Delta_\Gamma w & = 0 \\
    \label{eq:ch-u}
    - \eps \Delta_\Gamma u + \frac1\eps \psi'(u) - w & = 0.
  \end{align}
\end{subequations}
This leads to the fourth order Cahn-Hilliard equation on $\G_T$:
\begin{equation}
  \label{eq:ch}
  \md u + u \nabla_\Gamma \cdot v  = \Delta_\Gamma \left( - \eps \Delta_\Gamma u + \frac1\eps \psi'(u) \right).
\end{equation}
We close the system with the initial condition
\begin{equation}
  \label{eq:ch-init}
  u( \cdot, 0 ) = u_0 \quad \mbox{ on } \Gamma_0.
\end{equation}
There are no boundary conditions since the boundary of $\Gamma(t)$ is empty.

\begin{remark}
  One can derive the Cahn-Hilliard equations posed in a Cartesian domain as an $H^{-1}$ gradient flow of the Ginzburg-Landau functional. To obtain a gradient flow on an evolving surface, there would need to be a model for $v$ and which would lead to a coupled system for $u$ and $v$. In terms of modelling, we feel these extra terms are geometric terms determining an evolution equation for the surface, which we assume is given. Therefore, we do not consider such terms in this work.
\end{remark}

\subsection{Solution spaces}

In standard parabolic theory one looks for solutions in Bochner spaces. Considering our Cahn-Hilliard equation on a Cartesian domain $\Omega$ \cite{Ell89}, one would expect solutions to live in the spaces
\begin{equation*}
  u \in L^\infty(0,T;H^1(\Omega)), u' \in L^2(0,T;H^{-1}(\Omega)), w \in L^2(0,T;H^1(\Omega)).
\end{equation*}
These spaces are constructed by considering $u$ as a function from $(0,T)$ into the Hilbert space $H^1(\Omega)$. We would like to extend this definition so that $u(t)$ is in the now time-dependent Hilbert space $H^1(\Gamma(t))$. We consider Sobolev spaces over the space-time domain $\G_T$. We will write $\nabla_{\G_T}$ for the space-time gradient and $\mathrm{d} \sigma_T$ for the space-time measure on $\G_T$. This approach is similar to the Eulerian formulation of \cite{OlsReuXu13-pp}. We contrast our approach with that of \cite{Vie11-pp}, who proposed using an equivalent formulation using a reference domain.

We start by presenting the space-time domains $L^2(\G_T)$ and $H^1(\G_T)$ defined by
\begin{align*}
  L^2(\G_T) & := \left\{ \eta \in L^1_{\mathrm{loc}} (\G_T) : \int_{\G_T} \eta^2 \dd \sigma_T < + \infty \right\} \\
  H^1(\G_T) & := \Big\{ \eta \in L^2(\G_T) : \nabla_{\G_T} \eta \in L^2(\G_T) \Big\}.
\end{align*}
with norms
\begin{align*}
  \norm{ \eta }_{L^2(\G_T)} & := \left( \int_{\G_T} \eta^2 \dd \sigma_T \right)^{\frac12} \\
  \norm{ \eta }_{H^1(\G_T)} & := \left( \norm{ \eta }_{L^2(\G_T)}^2 + \norm{ \nabla_{\G_T} \eta }_{L^2(\G_T)}^2 \right)^{\frac12}.
\end{align*}
 
\begin{proposition}
  [\mycite{Theorem 2.9}{Heb00}]
  \label{prop:compact-embed}
  The space $H^1(\G_T)$ is compactly embedded into $L^2(\G_T)$.
\end{proposition}

Using the identities,
\begin{equation*}
  \int_0^T \int_{\Gamma(t)} \eta \dd \sigma \dd t = \int_{\G_T} \frac{\eta}{\sqrt{ 1 + \abs{ v_\nu }^2 }} \dd \sigma_T,
\end{equation*}
and
\begin{equation*}
  \nabla_{\G_T} \eta = \left( \nabla_\Gamma \eta + \frac{\nd \eta \, v_\nu }{ 1 + \abs{ v_\nu }^2 }, \frac{\nd \eta }{1 + \abs{v_\nu}^2 } \right),
\end{equation*}
our assumptions on $v$ imply that the space-time norms can be replaced with the equivalent norms
\begin{align*}
  \norm{ \eta }^\prime_{L^2(\G_T)} & := \left( \int_0^T \int_{\Gamma(t)} \eta^2 \dd \sigma \dd t \right)^{\frac12} \\
  \norm{ \eta }^\prime_{H^1(\G_T)} & := \left( \int_0^T \int_{\Gamma(t)} \eta^2 + \abs{ \nabla_\Gamma \eta }^2 + (\nd \eta)^2 \dd \sigma \dd t \right)^{\frac12}.
\end{align*}
We will use the equivalent primed norms (dropping the prime) on $L^2(\G_T)$ and $H^1(\G_T)$ in the following.

We define the space $L^2_{L^2}$ by
\begin{align*}
  L^2_{L^2} := \left\{ \eta \in L^1_{\mathrm{loc}}(\G_T) : \int_0^T \int_{\Gamma(t)} \eta \dd \sigma \dd t < + \infty \right\},
\end{align*}
with the inner product
\begin{equation*}
  ( \eta, \xi )_{L^2_{L^2}} := \int_0^T \int_{\Gamma(t)} \eta \xi \dd \sigma \dd t.
\end{equation*}
It is clear that $L^2_{L^2}$ is equivalent to $L^2(\G_T)$ and hence is a Hilbert space.

Next, we define the space $L^2_{H^1}$ as
\begin{align*}
  L^2_{H^1} := \left\{ \eta \colon \G_T \to \R : \eta \in L^2_{L^2} \mbox{ and } \nabla_\Gamma \eta \in (L^2_{L^2})^{n+1} \right\},
\end{align*}
with the inner product
\begin{equation*}
  ( \eta, \xi )_{L^2_{H^1}} := \int_0^T \int_{\Gamma(t)} \nabla_\Gamma \eta \cdot \nabla_\Gamma \xi + \eta \xi \dd \sigma \dd t,
\end{equation*}
where $\nabla_\Gamma \eta$ should be interpreted in the weak sense. Notice that elements of this space are weakly differentiable at almost every time.

\begin{lemma}
  The space $L^2_{H^1}$ is a Hilbert space.
\end{lemma}

\begin{proof}
  It is clear that $L^2_{H^1}$ is an inner product space and we are left to show completeness. Let $\eta_k$ be a Cauchy sequence in $L^2_{H^1}$. This implies that $\eta_k$ and $\nabla_\Gamma \eta_k$ are Cauchy sequences in $L^2(\G_T)$ and $( L^2(\G_T) )^{n+1}$. This means that there exists $\eta \in L^2(\G_T), \xi \in ( L^2(\G_T) )^{n+1}$ such that
  \begin{equation*}
    \norm{ \eta_k - \eta }_{L^2(\G_T)} + \norm{ \nabla_\Gamma \eta_k - \xi }_{L^2(\G_T)} \to 0 \quad \mbox{ as } k \to \infty.
  \end{equation*}
  Fix $t^* \in (0,T)$ and let $\varphi \in C^1(\Gamma(t^*))$ and $\alpha \in C(0,T)$. Using \myref[Lemma]{lem:pull-back}, we can construct $\tilde\varphi \colon \G_T \to \R$ such that $\tilde\varphi(\cdot, t) = \varphi$ and $\tilde\varphi \in C^1(\Gamma(t))$ for each time $t \in (0,T)$. Then, for $j=1,\ldots,n+1$, we obtain
  \begin{align*}
    & \int_0^T \int_{\Gamma(t)} \eta \underline{D}_j ( \alpha \tilde\varphi ) + \xi_j (\alpha \tilde\varphi) \dd \sigma \dd t \\
    & \quad = \int_0^T \int_{\Gamma(t)} ( \eta - \eta_k ) \underline{D}_j  (\alpha \tilde\varphi) + \big( \eta_k \underline{D}_j (\alpha \tilde\varphi) + \xi_j (\alpha \tilde\varphi) \big) \dd \sigma \dd t \\
    & \quad = \int_0^T \int_{\Gamma(t)} ( \eta - \eta_k ) \underline{D}_j (\alpha \tilde\varphi) + ( - \underline{D}_j \eta_k + \xi_j ) (\alpha \tilde\varphi) \dd \sigma \dd t,
  \end{align*}
  where we have used the fact that $\eta_k$ is weakly differentiable at almost every time. Taking the limit $k \to \infty$, we infer
  \begin{equation*}
    \int_0^T \alpha \left( \int_{\Gamma(t)} \eta \underline{D}_j \tilde\varphi + \xi_j \tilde\varphi \dd \sigma \right) \dd t = 0.
  \end{equation*}
  Since this holds for all $\alpha \in C(0,T)$, by the Fundamental Lemma of the Calculus of Variations, at $t = t^*$, we have
  \begin{equation*}
    \int_{\Gamma(t^*)} \eta \underline{D}_j \varphi + \xi_j \varphi \dd \sigma = 0 \quad \mbox{ for all } \varphi \in C^1(\Gamma(t^*)).
  \end{equation*}
  Since the choice of $t^*$ was arbitrary, we infer that $\xi$ is the weak gradient of $\eta$ for almost every time $t \in (0,T)$ and the proof is complete. \qed
\end{proof}

The equivalence of norms implies that $\eta \in L^2_{H^1}$ with $\md \eta \in L^2_{L^2}$ if, and only if, $\eta \in H^1(\G_T)$.

For $1 \le q \le \infty$, we will define the space $L^q_{H^1}$ by
\begin{equation*}
  L^q_{H^1}:= \left\{ \eta \in L^q(\G_T) : \norm{ \eta }_{L^q_{H^1}} < + \infty \right\},
\end{equation*}
with norm
\begin{equation*}
  \norm{ \eta }_{L^q_{H^1}} := 
  \begin{cases}
    \left( \displaystyle\int_0^T \norm{ \eta }_{H^1(\Gamma(t))}^q \dd t \right)^{\frac1q} & \quad \mbox{ for } q < \infty, \\
    \esssup\limits_{t \in (0,T)} \norm{ \eta }_{H^1(\Gamma(t))} & \quad \mbox{ for } q = \infty.
  \end{cases}
\end{equation*}
It is clear that $L^\infty_{H^1} \subset L^2_{H^1}$ and that
\begin{equation*}
  \norm{ \eta }_{L^2_{H^1}} \le \sqrt{T} \norm{ \eta }_{L^\infty_{H^1}} \quad \mbox{ for all } \eta \in L^\infty_{H^1}.
\end{equation*}
Finally, we define $L^\infty_{H^2}$ and $L^2_{H^2}$ by
\begin{align*}
  L^\infty_{H^2} & := \left\{ \eta \in L^2(\G_T) : \esssup_{t \in (0,T)} \norm{ \eta }_{H^2(\Gamma(t))} < + \infty \right\} \\
  L^2_{H^2} & := \left\{ \eta \in L^2(\G_T) : \int_0^T \norm{ \eta }_{H^2(\Gamma(t))}^2 \dd t < + \infty \right\}.
\end{align*}

\begin{remark}
  As a restriction on our analysis we will only consider $\md u$ as a function in  $L^2_{L^2}$ since we do not wish to consider a weak material derivative. Such considerations are left to future work.
\end{remark}

We conclude this section with a result which will take an integral in time equality into an almost everywhere in time equality. The proof is the generalisation of a similar result given in \cite[Lemma~7.4]{Rob01} for planar domains.

\begin{lemma}
  \label{lem:tiredy}
  Let $\eta \in L^2_{H^1}$ with
  \begin{equation}
    \int_0^T \int_{\Gamma(t)} \nabla_\Gamma \eta \cdot \nabla_\Gamma \xi + \eta \xi \dd \sigma \dd t = 0 \quad \mbox{ for all } \xi \in L^2_{H^1}.
  \end{equation}
  Then for almost all times $t \in (0,T)$, 
  \begin{equation}
    \label{eq:59}
    \int_{\Gamma(t)} \nabla_\Gamma \eta \cdot \nabla_\Gamma \varphi + \eta \varphi \dd \sigma = 0 \quad \mbox{ for all } \varphi \in L^2_{H^1}.
  \end{equation}
\end{lemma}

\begin{proof}
  Fix $\varphi \in L^2_{H^1}$ and $\alpha \in C([0,T])$, then choosing $\xi = \alpha \varphi \in L^2_{H^1}$ and
  \begin{equation*}
    0 = \int_0^T \int_{\Gamma(t)} \nabla_\Gamma \eta \cdot \nabla_\Gamma \xi + \eta \xi \dd \sigma \dd t
    = \int_0^T \alpha \left( \int_{\Gamma(t)} \nabla_\Gamma \eta \cdot \nabla_\Gamma \varphi + \eta \varphi \dd \sigma \right) \dd t.
  \end{equation*}
  Since the choice of $\alpha$ was arbitrary, the Fundamental Lemma of the Calculus of Variations implies the result. \qed
\end{proof}

\subsection{Weak and variational form}

We start by multiplying (\ref{eq:ch-w}, \ref{eq:ch-u}) by a test function $\varphi$ and apply integration by parts to the Laplacian terms to give the weak form. This will be the definition of solution used throughout this paper. Existence and uniqueness of solutions will be shown \myref[Section]{sec:well-posedn-cont}.

\begin{definition}[Weak solution]
  \label{def:ch-solution}
  We say that the pair $(u,w) \colon \G_T \to \R^2$, with $u \in L^\infty_{H^1} \cap H^1(\G_T)$ and $w \in L^2_{H^1}$, are a weak solution of the Cahn-Hilliard equation \eqref{eq:ch} if, for almost every time $t \in (0,T)$,
  \begin{subequations}
    \label{eq:ch-weak}
    \begin{align}
      \label{eq:ch-weak-w}
      \int_{\Gamma(t)} \md u \varphi + u \varphi \nabla_\Gamma \cdot v + \nabla_\Gamma w \cdot \nabla_\Gamma \varphi \dd \sigma & = 0 \\
      \label{eq:ch-weak-u}
      \int_{\Gamma(t)} \eps \nabla_\Gamma u \cdot \nabla_\Gamma \varphi + \frac1\eps \psi'(u) \varphi - w \varphi \dd \sigma & = 0,
    \end{align}
  \end{subequations}
  \begin{flushright}for all $\varphi \in L^2_{H^1}$,\end{flushright}
  and $u(\cdot, 0) = u_0$ pointwise almost everywhere in $\Gamma_0$.
\end{definition}

Restricting our thoughts to $\varphi \in H^1(\G_T)$, applying the transport formula to the first two terms in \eqref{eq:ch-weak-w} gives the variational formulation:
\begin{subequations}
  \label{eq:ch-var}
  \begin{align}
    \label{eq:ch-var-w}
    \dt \left( \int_{\Gamma(t)} u \varphi \dd \sigma \right) + \int_{\Gamma(t)} \nabla_\Gamma w \cdot \nabla_\Gamma \varphi \dd \sigma & = \int_{\Gamma(t)} u \md \varphi \dd \sigma \\
    \label{eq:ch-var-u}
    \int_{\Gamma(t)} \eps \nabla_\Gamma u \cdot \nabla_\Gamma \varphi + \frac1\eps \psi'(u) \varphi \dd \sigma  & = \int_{\Gamma(t)} w \varphi \dd \sigma.
  \end{align}
\end{subequations}
We remark that this formulation has no explicit mention of the velocity field $v$ and will be the basis of our finite element calculations.

It will be useful to write these equations using abstract bilinear forms. We define the following three to describe the above equations for $\eta, \varphi \in H^1(\Gamma(t))$:
\begin{equation*}
  \begin{gathered}
    m ( \eta, \varphi ) =
    \int_{\Gamma(t)} \eta \varphi \dd \sigma \qquad
    a ( \eta, \varphi ) = 
    \int_{\Gamma(t)} \nabla_\Gamma \eta \cdot \nabla_\Gamma \varphi \dd \sigma \\
    g ( v; \eta, \varphi ) =
    \int_{\Gamma(t)} \eta \varphi \nabla_\Gamma \cdot v \dd \sigma.
  \end{gathered}
\end{equation*}
This lets us write \eqref{eq:ch-weak} as
\begin{equation}
  \label{eq:ch-weak-abs}
  \begin{aligned}
    m ( \md u, \varphi ) + g ( v; u, \varphi ) + a ( w, \varphi ) & = 0 \\
    \eps a ( u, \varphi ) + \frac1\eps m ( \psi'(u), \varphi ) - m ( w, \varphi ) & = 0,
  \end{aligned}
\end{equation}
and \eqref{eq:ch-var} as
\begin{equation}
  \label{eq:ch-var-abs}
  \begin{aligned}
    \dt m ( u, \varphi ) + a ( w, \varphi ) & = m ( u, \md \varphi ) \\
    \eps a ( u, \varphi ) + \frac1\eps m ( \psi'(u), \varphi ) & = m ( w, \varphi ).
  \end{aligned}
\end{equation}

We may also write the results of \myref[Lemma]{lem:transport} in this form:
\begin{align*}
  \dt m ( \eta, \varphi ) & = m ( \md \eta, \varphi ) + m ( \eta, \md \varphi ) + g ( v; \eta, \varphi ) \\
  \dt a ( \eta, \varphi ) & = a ( \md \eta, \varphi ) + a ( \eta, \md \varphi ) + b ( v; \eta, \varphi ),
\end{align*}
with the addition of
\begin{equation*}
  b ( v; \eta, \varphi ) = \int_{\Gamma(t)} \B(v) \nabla_\Gamma \eta \cdot \nabla_\Gamma \varphi \dd \sigma,
\end{equation*}
using $\A = \id$ in the definition of $\B(v)$.

%--- finite element approximation -------------------------------------------%           
                                                                                         
\section{Finite element approximation}
\label{sec:finite-elem-appr}

In this section, we propose a finite element method for approximating solutions of the Cahn-Hilliard equation \eqref{eq:ch} based on the evolving surface finite element method \cite{DziEll07}.

\subsection{Evolving triangulation and discrete material derivative}

Let $\Gamma_{h,0}$ be  a polyhedral approximation of the initial surface $\Gamma_0$ with the restriction that the nodes $\{ X_j^0 \}_{j=1}^N$ of $\Gamma_{h,0}$ lie on $\Gamma_0$. We evolve the nodes $\{ X_j(t) \}_{j=1}^N$ by the smooth surface velocity:
\begin{equation*}
  \dot{X}_j(t) = v( X_j(t), t ), \quad X_j(0) = X_j^0, \quad \mbox{ for } j = 1, \ldots, N.
\end{equation*}
Linearly interpolating between these nodes defines a family of discrete surfaces $\{ \Gamma_h(t) \}$. At each time, we assume that we have a triangulation $\T_h(t)$ of $\Gamma_h(t)$, with $h$ the maximum diameter of elements in $\T_h(t)$ uniformly in time:
\begin{equation}
  \label{eq:66}
  h := \sup_{t \in (0,T)} \max_{E(t) \in \T_h(t)} \diam\, E(t).
\end{equation}
We assume this triangulation is quasi-uniform \cite{BreSco02} uniformly in time.

\begin{remark}
  In practical situations, assuming a uniformly regular mesh may not be feasible. Large surface deformations can lead to poor quality triangulations with deformed elements. In such cases, re-meshing may be required \cite{ClaDieDzi04,EilEll08}. Alternatively, one may use an arbitrary Lagrangian-Eulerian formulation by allowing extra tangential mesh motions \cite{EllSti13,EllSty12}.
\end{remark}

We define $\nu_h$ element-wise as the unit outward pointing normal to $\Gamma_h(t)$ and denote by $\nabla_{\Gamma_h}$ the tangential gradient on $\Gamma_h(t)$ defined element-wise by
\begin{equation*}
  \nabla_{\Gamma_h} \eta_h := \nabla \tilde\eta_h - ( \nabla \tilde\eta_h \cdot \nu_h ) \nu_h = ( \id - \nu_h \otimes \nu_h ) \nabla \tilde\eta_h =: P_h \nabla \tilde\eta_h.
\end{equation*}
This is a vector-valued quantity and we will denote its components by
\begin{equation*}
  \nabla_{\Gamma_h} \eta_h = \left( \underline{D}_{h,1} \eta_h, \ldots, \underline{D}_{h,n+1} \eta_h \right)
\end{equation*}

We define the finite element space of piecewise linear functions on $\Gamma_h(t)$ by
\begin{equation}
  \label{eq:68}
  S_h(t) := \{ \phi_h \in C( \Gamma_h(t) ) : \phi_h |_{E(t)} \mbox{ is affine linear, for each } E(t) \in \T_h(t) \}.
\end{equation}
We will write $\{ \phi_j^N ( \cdot, t ) \}_{j=1}^N$ for the nodal basis of $S_h(t)$ given by $\phi_j^N( X_i(t), t ) = \delta_{ij}$.

The definition of a basis of $S_h(t)$ allows us to characterise the velocity of the surface $\{ \Gamma_h(t) \}$. An arbitrary point $X(t)$ on $\Gamma_h(t)$ evolves according to the discrete velocity $V_h$ given by
\begin{equation}
  \label{eq:69}
  \dot{X}(t) = V_h( X(t), t) := \sum_{j=1}^N \dot{X}_j(t) \phi_j^N( X(t), t )
  = \sum_{j=1}^N v( X_j(t), t ) \phi_j^N( X(t), t ).
\end{equation}

We will write $\G_{h,T}$ as the discrete equivalent to $\G_T$:
\begin{equation}
  \label{eq:63}
  \G_{h,T} := \bigcup_{t \in (0,T)} \Gamma_h(t) \times \{ t \}.
\end{equation}
The discrete velocity $V_h$ induces a discrete material derivative. For a scalar quantity $\eta_h$ on $\G_{h,T}$, we define the discrete material derivative $\md_h \eta_h$ by
\begin{equation}
  \label{eq:70}
  \md_h \eta_h := \partial_t \tilde\eta_h + \nabla \tilde\eta_h \cdot V_h,
\end{equation}
where $\tilde\eta_h$ is an arbitrary extension of $\eta_h$ to $\N_T$. This leads to the remarkable transport property of the basis functions $\{ \phi_j^N \}$.
\begin{lemma}
  [Transport of basis functions \mycite{Proposition~5.4}{DziEll07}]
  \label{lem:transport-basis}
  Let $\phi_j^N\colon \G_{h,T} \to \R$ be a nodal basis function as described above, then
  \begin{equation}
    \label{eq:71}
    \md_h \phi_j^N = 0.
  \end{equation}
\end{lemma}

From a practical view point, a key advantage of this methodology is that, since basis functions have zero discrete material velocity, there is no mention of the velocity or curvature in the resulting finite element scheme.

These discrete quantities also satisfy a variant of the transport formula from \myref[Lemma]{lem:transport}. We label the surface measure on $\Gamma_h(t)$ as $\mathrm{d} \sigma_h$.

\begin{lemma}
  [Transport lemma for triangulated surfaces \mycite{Lemma~4.2}{DziEll13}]
  \label{lem:discrete-transport}
  Let $\{ \Gamma_h(t) \}$ be a discrete family of triangulated surfaces evolving with velocity $V_h$. Let $\eta_h, \phi_h$ be time-dependent finite element functions such that the following quantities exist. Then, we have
  \begin{equation}
    \label{eq:discrete-transport}
    \dt \int_{\Gamma_h(t)} \eta_h \dd \sigma_h
    = \int_{\Gamma_h(t)} \md_h \eta_h + \eta_h \nabla_{\Gamma_h} \cdot V_h \dd \sigma_h.
  \end{equation}
  In particular, for the $L^2$ inner product this means that
  \begin{equation}
    \label{eq:mh-transport}
    \dt \int_{\Gamma_h(t)} \eta_h \phi_h \dd \sigma_h
    = \int_{\Gamma_h(t)} ( \md_h \eta_h ) \phi_h  + \eta_h ( \md_h \phi_h ) +
    \eta_h \phi_h \nabla_{\Gamma_h} \cdot V_h \dd \sigma_h,
  \end{equation}
  and for the Dirichlet inner product, we obtain
  \begin{equation}
    \label{eq:ah-transport}
    \begin{aligned}
      & \dt \int_{\Gamma_h(t)} \nabla_{\Gamma_h} \eta_h \cdot \nabla_{\Gamma_h} \phi_h \dd \sigma_h \\
      & \qquad = \int_{\Gamma_h(t)} \nabla_{\Gamma_h} ( \md_h \eta_h ) \cdot \nabla_{\Gamma_h} \phi_h + \nabla_{\Gamma_h} \eta_h \cdot \nabla_{\Gamma_h} ( \md_h \phi_h ) \dd \sigma_h \\
      & \qquad \quad + \sum_{E(t) \in \T_h(t)} \int_{E(t)} \B_h( V_h ) \nabla_{\Gamma_h} \eta_h \cdot \nabla_{\Gamma_h} \phi_h \dd \sigma_h,
    \end{aligned}
  \end{equation}
  where
  \begin{equation*}
    \B_h( V_h ) = \frac12 ( \nabla_{\Gamma_h} \cdot V_h ) \id - D_h( V_h ) \quad \mbox{ and } \quad 
    D_h( V_h )_{ij} = \frac12 \big( \underline{D}_{h,i} V_{h,j} + \underline{D}_{h,j} V_{h,i} \big).
  \end{equation*}
\end{lemma}

\begin{lemma}
  \label{lem:vh-bound}
  Under our assumptions on $\{ \Gamma_h(t) \}$, we have that
  \begin{equation}
    \label{eq:106}
    \sup_{t \in [0,T]} \left( \norm{ \nabla_{\Gamma_h} \cdot V_h }_{L^\infty(\Gamma_h(t))} + \norm{ \B_h(V_h) }_{L^\infty(\Gamma_h(t))} \right) \le c \sup_{t \in [0,T]} \norm{ v }_{C^2(\N_T)}.
  \end{equation}
\end{lemma}

\begin{proof}
  The result follows from applying the geometric estimates \eqref{eq:79} and \eqref{eq:92} along with our assumption that $v \in C^2(\N_T)$. \qed
\end{proof}

\subsection{Finite element scheme}

We will assume that there exists a mesh size $h_0 > 0$ such that $\norm{ U_0 }_{H^1(\Gamma_{h,0})}$ is bounded independently of $h$ for $h < h_0$. This implies that there exists $C > 0$ such that for all $h < h_0$, we have
\begin{equation}
  \label{eq:ch-init-assumption}
  \E_0^h := \int_{\Gamma_{h,0}} \frac{\eps}{2} \abs{ \nabla_{\Gamma_h} U_0 }^2  + \frac1\eps \psi( U_0 ) \dd \sigma_h < C.
\end{equation}

\begin{remark}
  One particular choice of initial condition will be to take $U_0$ as a suitable approximation  of $u_0$ (for example, $\Pi_h u_0$ defined in \eqref{eq:102}) for $u_0 \in H^2(\Gamma_0)$.
\end{remark}

Our solution spaces will be
\begin{equation}
  \label{eq:ShT-defn}
  \begin{aligned}
    \tilde{S}_h^T & := \{ \phi_h  \in C( \G_{h,T} ) : \phi_h( \cdot, t ) \in S_h(t) \mbox{ for all } t \in [0,T] \} \\
    S_h^T & := \{ \phi_h  \in \tilde{S}_h^T : \md_h \phi_h \in C( \G_{h,T} ) \}.
  \end{aligned}
\end{equation}

The finite element scheme is: Given $U_0$, find $U_h \in S_h^T$ and $W_h \in \tilde{S}_h^T$ such that for almost every time $t \in (0,T)$
\begin{subequations}
  \label{eq:ch-fem}
  \begin{align}
    \label{eq:ch-fem-w}
    \dt \left( \int_{\Gamma_h(t)} U_h \phi_h \dd \sigma_h \right)
    + \int_{\Gamma_h(t)} \nabla_{\Gamma_h} W_h \cdot \nabla_{\Gamma_h} \phi_h \dd \sigma_h
    & = \int_{\Gamma_h(t)} U_h \md_h \phi_h \dd \sigma_h \\
    \label{eq:ch-fem-u}
    \int_{\Gamma_h(t)} \eps \nabla_{\Gamma_h} U_h \cdot \nabla_{\Gamma_h} \phi_h 
    + \frac1\eps \psi'(U_h) \phi_h \dd \sigma_h
    & = \int_{\Gamma_h(t)} W_h \phi_h \dd \sigma_h
  \end{align}
\end{subequations}
\begin{flushright}for all $\phi_h \in S_h(t)$,\end{flushright}
subject to the initial condition
\begin{equation}
  \label{eq:67}
  U_h( \cdot, 0 ) = U_0.
\end{equation}

The transport formula \eqref{eq:mh-transport} implies that, for $\phi_h \in S_h^T$, \eqref{eq:ch-fem-w} is equivalent to
\begin{equation}
  \label{eq:ch-fem-w'}
  \int_{\Gamma_h(t)} \md_h U_h \phi_h + U_h \phi_h \nabla_{\Gamma_h} \cdot V_h
  + \nabla_{\Gamma_h} W_h \cdot \nabla_{\Gamma_h} \phi_h \dd \sigma_h = 0.
\end{equation}

We can write these equations in matrix form. First, we will introduce vectors $\alpha(t), \beta(t) \in \R^{N}$ for the nodal values of $U_h$ and $W_h$ by
\begin{equation*}
  U_h( x, t ) = \sum_{j=1}^N \alpha_j(t) \phi_j^N( x, t ),
  \quad\!
  W_h( x, t ) = \sum_{j=1}^N \beta_j(t) \phi_j^N( x, t )
  \quad\!
  \mbox{ for } (x,t) \in \G_{h,T}.
\end{equation*}
In place of the bilinear forms, we have the mass matrix $\M(t)$ and stiffness matrix $\S(t)$:
\begin{align*}
  \M(t)_{ij} = \int_{\Gamma_h(t)} \phi_i^N \phi_j^N \dd \sigma_h 
  \qquad
  \S(t)_{ij} = \int_{\Gamma_h(t)} \nabla_{\Gamma_h} \phi_i^N \cdot \nabla_{\Gamma_h} \phi_j^N \dd \sigma_h,
\end{align*}
and in place of the non-linear term, we will write 
\begin{align*}
  \Psi( \alpha(t) )_j = \int_{\Gamma_h(t)} \psi' \left( \sum_{i=1}^N \alpha_i(t) \phi_i^N \right) \phi_j^N \dd \sigma_h.
\end{align*}
Using the transport of basis property (\myref[Lemma]{lem:transport-basis}), we can write \eqref{eq:ch-fem} as
\begin{subequations}
  \label{eq:ch-matrix}
  \begin{align}
    \dt \left( \M(t) \alpha(t) \right) + \S(t) \beta(t) & = 0 \\
    \eps \S(t) \alpha(t) + \frac1\eps \Psi( \alpha(t) ) - \M(t) \beta(t) & = 0.
  \end{align}
\end{subequations}
Alternatively, eliminating $\beta(t)$, this can be written as
\begin{align}
  \label{eq:ch-matrix-lumped}
  \dt \left( \M(t) \alpha(t) \right) + \S(t) \M(t)^{-1} \left( \eps \S(t) \alpha(t) + \frac1\eps \Psi( \alpha(t) ) \right ) = 0.
\end{align}
One could also use lumped mass integration \cite[Chapter 15]{Tho06} instead of the full mass matrix.

Notice that this is the same structure as a finite element discretisation of a Cahn-Hilliard equation posed on a planar domain. We now have time dependent matrices which need to be assembled on each time step. Various time stepping schemes have been considered for second-order parabolic problems on evolving surfaces \cite{DziEll12,DziLubMan11,LubManVen13}.

Next, we introduce abstract notation which permit a more compact writing of the analysis that follows:
\begin{gather*}
  m_h ( \eta_h, \phi_h ) =
  \int_{\Gamma_h(t)} \eta_h \phi_h \dd \sigma_h \qquad
  a_h ( \eta_h, \phi_h ) = 
  \int_{\Gamma_h(t)} \nabla_{\Gamma_h} \eta_h \cdot \nabla_{\Gamma_h} \phi_h \dd \sigma_h \\
  g_h ( V_h; \eta_h, \phi_h ) =
  \int_{\Gamma_h(t)} \eta_h \phi_h \nabla_{\Gamma_h} \cdot V_h \dd \sigma_h.
\end{gather*}
This lets us write \eqref{eq:ch-fem} as
\begin{align*}
 \dt m_h ( U_h, \phi_h ) + a_h ( W_h, \phi_h ) & = m_h ( U_h, \md_h \phi_h ) \\
  \eps a_h ( U_h, \phi_h ) + \frac1\eps m_h ( \psi'(U_h), \phi_h ) & = m_h ( W_h, \phi_h ),
\end{align*}
and \eqref{eq:ch-fem-w'} as
\begin{equation*}
  m_h( \md_h U_h, \phi_h ) + g_h( V_h; U_h, \phi_h ) + a_h( W_h, \phi_h ) = 0.
\end{equation*}
The transport laws from \myref[Lemma]{lem:discrete-transport} transfer to the abstract setting also:
\begin{align*}
  \dt m_h (\eta_h, \phi_h )
  & = m_h ( \md_h \eta_h, \phi_h ) +  m_h ( \eta_h, \md_h \phi_h ) 
  + g_h ( V_h; \eta_h, \phi_h ) \\
  \dt a_h ( \eta_h, \phi_h )
  & = a_h ( \md_h \eta_h, \phi_h ) +  a_h ( \eta_h, \md_h \phi_h ) 
  + b_h ( V_h; \eta_h, \phi_h ),
\end{align*}
where
\begin{equation*}
  b_h( V_h; \eta_h, \phi_h ) = \sum_{E(t) \in \T_h(t)} \int_{E(t)} \B_h( V_h ) \nabla_{\Gamma_h} \eta_h \cdot \nabla_{\Gamma_h} \phi_h \dd \sigma_h.
\end{equation*}

Under the above assumptions, the following estimates are possible.

\begin{theorem}
  [Well-posedness of the finite element scheme \eqref{eq:ch-fem}]
  \label{thm:ch-fem-wp}
  Under the above assumptions on $U_0$ and $\{ \Gamma_h(t) \}$, there exists a unique solution pair $(U_h, W_h) \in S_h^T \times \tilde{S}_h^T$, both with $C^1$ in time nodal values, to the finite element scheme \eqref{eq:ch-fem} and $\int_{\Gamma_h(t)} U_h \dd \sigma_h$ is conserved:
  \begin{equation}
    \label{eq:mass-conserv}
    \int_{\Gamma_h(t)} U_h \dd \sigma_h = \int_{\Gamma_{h,0}} U_0 \dd \sigma_h \quad \mbox{ for all } t \in (0,T).
  \end{equation}
  Furthermore, there exists $h_1$, $0 < h_1 < h_0$, and $C_0 >0$, which depend on the final time $T$ and the $H^1(\Gamma_{h,0})$-norm of the initial condition $U_0$, such that for all $h < h_1$ the following bound is satisfied:
  \begin{equation}
    \label{eq:72}
    \begin{aligned}
      \sup_{t \in (0,T)} \int_{\Gamma_h(t)} \frac\eps2 \abs{ \nabla_{\Gamma_h} U_h }^2 + \frac1\eps \psi( U_h ) \dd \sigma_h + \frac12 \int_0^T \norm{ \nabla_{\Gamma_h} W_h }_{L^2(\Gamma_h(t))}^2 \dd t \le C_0.
    \end{aligned}
  \end{equation}
\end{theorem}

The proof will be shown after we have proven some intermediate results.

\subsection{Lifted finite elements}

The following analysis will rely on lift operators defined using a time dependent closest point operator $p$ \eqref{eq:101}. This lifting process will also be applied to the surface triangulation. This will induce a further discrete material velocity $v_h$ which will describe how the lifts of triangles on $\{\Gamma(t)\}$ evolve.

First, for a function $\eta_h \colon \G_{h,T} \to \R$, we define its lift, $\eta_h^\ell \colon \G_T \to \R$, implicitly, by:
\begin{equation}
  \label{eq:77}
  \eta_h^\ell( p(x,t), t ) = \eta_h( x, t ),
\end{equation}
and, for a function $\eta \colon \G_T \to \R$, we define its inverse lift, $\eta^{-\ell} \colon \G_{h,T} \to \R$ by
\begin{equation}
  \label{eq:76}
  \eta^{-\ell}( x, t ) := \eta( p( x ,t ), t ).
\end{equation}
It is clear that these operations are inverses of each other
\begin{equation*}
  (\eta^{-\ell})^\ell = \eta \qquad \mbox{ and } \qquad
  ( \eta_h^\ell )^{-\ell} = \eta_h.
\end{equation*}

Furthermore, \eqref{eq:101} allows us to define a lifted triangulation $\T_h^\ell(t)$ of $\Gamma(t)$ by
\begin{equation}
  \label{eq:80}
  \T_h^\ell = \{ e(t) = E^\ell(t) : E(t) \in \T_h(t) \}, \qquad
  E^\ell(t) := \{ p(x,t) : x \in E(t) \}.
\end{equation}
This defines an exact triangulation of $\Gamma(t)$.

\begin{lemma}
  [Stability of lift \mycite{for $q=2$}{DziEll07}]
  \label{lem:ch-norm-equiv}
  Let $\eta_h \colon \G_{h,T} \to \R$, with lift $\eta_h^\ell \colon \G_T \to \R$, be such that the following quantities exist. For $1 \le q \le +\infty$, there exists $c_1, c_2 > 0$, independent of $h$, but depending on $q$, such that for each time $t \in [0,T]$ and each element $E(t) \in \T_h(t)$ with associated lifted element $e(t) \in \T_h^\ell(t)$, the following hold:
  \begin{subequations}
    \label{eq:79}
    \begin{align}
      c_1 \norm{ \eta_h^\ell }_{L^q(e(t))} 
      & \le \norm{  \eta_h }_{L^q(E(t))}
      \le c_2 \norm{ \eta_h^\ell }_{L^q(e(t))} \\
      \label{eq:79b}
      c_1 \norm{ \nabla_\Gamma \eta_h^\ell }_{L^q(e(t))} 
      & \le \norm{ \nabla_{\Gamma_h} \eta_h }_{L^q(E(t))}
      \le c_2 \norm{ \nabla_\Gamma \eta_h^\ell }_{L^q(e(t))} \\
      \norm{ \nabla_{\Gamma_h}^2 \eta_h }_{L^2(E(t))}
      & \le c \left( \norm{ \nabla_\Gamma^2 \eta_h^\ell }_{L^2(e(t))} + h \norm{ \nabla_\Gamma \eta_h^\ell }_{L^2(e(t))} \right).
    \end{align}
  \end{subequations}
\end{lemma}

This result allows us to give Sobolev embeddings for discrete surfaces:

\begin{lemma}
  \label{lem:sobolev-embedh}
  For $\Gamma_h(t)$ as above,
  \begin{align}
    W^{1,q}(\Gamma_h(t)) \subset
    \begin{cases} 
      L^{ nq / ( n-q )}( \Gamma_h(t) ) & \mbox{ for }  q < n \\
      L^\infty( \Gamma_h(t) ) & \mbox{ for } q > n.
    \end{cases}
  \end{align}
  Furthermore there exists a constant $c = c(n,q)$, independent of $h$, such that for any $\eta_h \in W^{1,q}(\Gamma_h(t))$
  \begin{subequations}
    \begin{align}
      \norm{ \eta_h }_{L^{n q / (n-q)}(\Gamma_h(t))} & \le c \norm{ \eta_h }_{W^{1,q}(\Gamma_h(t))} && \mbox{ for } q < n \\
      \norm{ \eta_h }_{L^\infty(\Gamma_h(t))} & \le c \norm{ \eta_h }_{W^{1,q}(\Gamma_h(t))} && \mbox{ for } q > n.
    \end{align}
  \end{subequations}
\end{lemma}

\begin{proof}
  To see the embedding result, we apply \myref[Lemma]{lem:sobolev-embed}. The bounds then follow using the stability of the lift (\myref[Lemma]{lem:ch-norm-equiv}). \qed
\end{proof}

We will write $S_h^\ell(t)$ for the space of lifted finite element functions:
\begin{equation*}
  S_h^\ell(t) = \{ \varphi_h = \phi_h^\ell : \phi_h \in S_h(t) \}.
\end{equation*}
This space comes with the standard approximation property:
\begin{proposition}
  [Approximation property]
  \label{prop:ch-approx}
  The Lagrangian interpolation operator $I_h \colon C(\Gamma(t)) \to S_h^\ell(t)$ is well defined and, for $z \in H^2(\Gamma(t))$, satisfies the bound
  \begin{equation}
    \label{eq:ch-approx}
    \norm{ z - I_h z }_{L^2(\Gamma(t))} + h \norm{ \nabla_\Gamma ( z - I_h z ) }_{L^2(\Gamma(t))} \le c h^2 \norm{ z }_{H^2(\Gamma(t))}.
  \end{equation}
  Let $1 \le q \le \infty$ be such that $H^1(\Gamma(t))$ embeds into $L^q(\Gamma(t))$, then
  \begin{equation}
    \label{eq:ch-approxp}
    \norm{ \nabla_\Gamma ( z - I_h z ) }_{L^q(\Gamma(t))} \le c h^{1 + \min(0,n/q - n/2)} \norm{ z }_{H^2(\Gamma(t))}.
  \end{equation}
\end{proposition}

\begin{proof}
  The proof is given in \cite{Dzi88} for the case $q=2$ and can be easily extended using standard interpolation theory \cite[Theorem~3.1.6]{Cia78} to the case $q \neq 2$. \qed
\end{proof}

\begin{remark}
  For the remainder of the paper, we will write lower case letters for the lift finite element functions with capital letters (i.e. $U_h^\ell = u_h$ and $W_h^\ell = w_h$) and $\varphi_h$ for the lift of $\phi_h$.
\end{remark}

The motion of the edges of the simplices in the triangulation $\{ \T_h^\ell(t) \}$ defines a discrete material velocity for the surface $\{ \Gamma(t) \}$. Let $X(t)$ be the trajectory of a point on $\{ \Gamma_h(t) \}$ with velocity $V_h(X(t),t)$. We set $Y(t) = p(X(t),t)$ then define $v_h$ by
\begin{equation}
  \label{eq:75}
  v_h( Y(t), t ) := \dot{Y}(t) = \frac{\partial p}{\partial t} ( X(t), t) + \nabla p( X(t), t ) \cdot V_h( X(t), t),
\end{equation}
so that for $x \in \Gamma_h(t)$, using \eqref{eq:101}, we have
\begin{equation*}
  v_h( p(x,t), t ) = ( P (x,t) - d (x,t)  \H( x,t ) ) V_h(x,t) - d_t(x,t) \nu(x) - d(x,t) \nu_t(x,t).
\end{equation*}
This defines another discrete material derivative for functions $\varphi_h( \cdot, t ) \in S_h^\ell(t)$. We define the discrete material derivative on $\G_T$ element-wise by
\begin{equation}
  \label{eq:78}
  \md_h \varphi_h := \partial_t \varphi_h + v_h \cdot \nabla \varphi_h.
\end{equation}
A quick calculation \cite{DziEll13} shows that for all $\phi_h \in S_h(t)$, with lift $\varphi_h \in S_h^\ell(t)$,
\begin{equation}
  \label{eq:81}
  \md_h \varphi_h = ( \md_h \phi_h )^\ell.
\end{equation}
It can be shown, similarly to \eqref{eq:71}, that $\md_h (\phi_j^N)^\ell = 0$. We will write $S_h^{\ell,T}$ and $\tilde{S}_h^{\ell,T}$ for the lifts of the spaces $S_h^T$ and $\tilde{S}_h^T$ defined by \eqref{eq:ShT-defn}. It is clear that from \myref[Lemma]{lem:ch-norm-equiv} that
\begin{equation*}
  S_h^{\ell,T} \subset H^1(\G_T) \quad \mbox{ and } \quad \tilde{S}_h^{\ell,T} \subset L^2_{H^1}.
\end{equation*}

We remark that the continuous and discrete material velocities on $\{\Gamma(t)\}$ only differ in the tangential direction. This implies that the difference between the two material derivatives on $\{ \Gamma(t) \}$ only depends on the tangential gradient of the original function and not on any time derivatives.

These definitions also permit transport formulae:
\begin{lemma}
  [Transport lemma for smooth triangulated surfaces \mycite{Lemma~4.2}{DziEll13}]
  Let $\{\Gamma(t)\}$ be an evolving surface decomposed at each time into a family curved elements $\{ \T_h^\ell(t) \}$ whose edges evolve with velocity $v_h$. Then the following relations hold for functions $\eta_h, \varphi_h \colon \G_T \to \R$ such that the following quantities exist:
  \begin{equation}
    \label{eq:82}
    \dt \int_{\Gamma(t)} \eta_h \dd \sigma 
    = \int_{\Gamma(t)} \md_h \eta_h + \eta_h \nabla_\Gamma \cdot v_h \dd \sigma,
  \end{equation}
  and
  \begin{align}
    \label{eq:83}
    \dt m( \eta_h, \varphi_h ) & = m( \md_h \eta_h, \varphi_h ) + m( \eta_h, \md_h \varphi_h ) + g( v_h; \eta_h, \varphi_h ) \\
    \dt a( \eta_h, \varphi_h ) & = a( \md_h \eta_h, \varphi_h ) + a( \eta_h, \md_h \varphi_h ) + b( v_h; \eta_h, \varphi_h ).
  \end{align}
\end{lemma}

\subsection{Proof of finite element scheme well-posedness}

Before showing stability of the finite element scheme, we will show a generalised Gronwall inequality:
\begin{lemma}
  Let $y_h(t), z_h(t) \ge 0$ and satisfy the following differential inequality for $\tilde{C} \ge 0$
  \begin{equation}
    \label{eq:6}
    \begin{aligned}
      \dt y_h(t) + z_h(t) & \le c \big( y_h(t) + h y_h(t)^2 + h^2 y_h(t)^3 + \tilde{C} \big) \quad \mbox{ for } 0 \le t \le T, \\
      y_h(0) & = y_0.
    \end{aligned}
  \end{equation}
  Let $h$ be sufficiently small so that $1 - h ( y_0 + \tilde{C} )^2 ( e^{2(1+h)ct} - 1 ) > 0$, then $y_h, z_h$ satisfy the bound
  \begin{equation}
    \label{eq:7}
    \begin{aligned}
      & y_h(t) + \int_0^t z_h( s ) \dd s
      \le \frac{ e^{(1+h) c t} ( y_0 + \tilde{C} ) }{ \sqrt{ 1 - h ( y_0 + \tilde{C} )^2 ( e^{2(1+h)ct} - 1 )  }}.
    \end{aligned}
  \end{equation}
\end{lemma}

\begin{proof}
  Let $\eta_h(t) = y_h(t) + \int_0^t z_h(s) \dd s + \tilde{C}$ and $\eta_0 = y_0 + \tilde{C}$. We note that $y_h(t)^q \le \eta_h(t)^q$ ($q =1,2,3$) and $\eta_h(t) \ge 0$. Then $\eta_h$ satisfies
  \begin{align*}
    \dt \eta_h(t) \le c \big( \eta_h(t) + h \eta_h(t)^2 + h^2 \eta_h(t)^3 \big)
    \le c \big( (1+h) \eta_h(t) + (h+h^2) \eta_h(t)^3 \big).
  \end{align*}
  This implies
  \begin{equation*}
    \frac{1}{2+2h} \dt \log \left( \frac{ \eta_h(t)^2 }{ 1 + h \eta_h(t)^2 } \right) = \frac{ \dt \eta_h(t) }{ (1+h) \eta_h(t) + (h+h^2) \eta_h(t)^3 }
    \le c.
  \end{equation*}
  Assuming that $h$ is sufficiently small so that $1 - h \eta_0^2 ( e^{2(1+h)ct} - 1 ) > 0$, integrating this inequality in time implies
  \begin{equation*}
    \eta_h(t)^2 \le \frac{ e^{2 (1+h) c t} \eta_0^2 }{ 1 - h \eta_0^2 \, ( e^{2(1+h)ct} - 1 ) }.
  \end{equation*}
  Rearranging this inequality gives the desired result. \qed
\end{proof}

We can now show the stability result in Theorem~\ref{thm:ch-fem-wp}.

\begin{proof}[Proof of Theorem~\ref{thm:ch-fem-wp}]
  Considering \eqref{eq:ch-matrix-lumped}, since $\M(t)$ is positive definite, $\S(t)$ positive semi-definite and $\Psi$ is locally Lipschitz, standard theory of ordinary differential equations gives a unique short-time solution $\alpha \in C^1( [0,T_0]; \R^N)$ for some $T_0 < T$. From \eqref{eq:discrete-transport}, we know $\S(t)$ and $\M(t)$ are $C^1$ in time, and $\M(t)^{-1} \in C^1$ by the Inverse Function Theorem. Thus, we infer
  \begin{equation*}
    \beta(t) = \M(t)^{-1} \S(t) \M(t)^{-1} ( \eps \S(t) \alpha(t) + \frac1\eps \Psi(\alpha(t)) ) \in C^1([0,T_0]; \R^N).
  \end{equation*}
  This is easily translated into solutions $U_h, W_h$ in the appropriate spaces.

  Since $\phi_h = 1$ is an admissible test function in \eqref{eq:ch-fem-w}, it is clear that $\int_{\Gamma_h(t)} U_h \dd \sigma_h$ is conserved. 
  
  To extend to the long-term solution, we construct an energy bound. We start by testing \eqref{eq:ch-fem-w} with $W_h$ and \eqref{eq:ch-fem-u} with $\md_h U_h$ and sum to see
  \begin{align*}
    & \eps a_h ( U_h, \md_h U_h ) + \frac1\eps m_h( \psi'(U_h), \md_h U_h) + a_h( W_h,  W_h ) \\
    & \quad = - \dt m_h( U_h, W_h ) + m_h( U_h, \md_h W_h ) + m_h( \md_h U_h, W_h ).
  \end{align*}
  Applying the transport formulae from \myref[Lemma]{lem:discrete-transport}, we obtain
  \begin{align*}
    & \dt \left( \eps a_h( U_h, U_h ) + \frac1\eps m_h( \psi(U_h), 1 ) \right) + a_h( W_h, W_h ) \\
    & \quad = \frac\eps2 b_h( V_h; U_h, U_h ) + \frac1\eps g_h( V_h; \psi(U_h), 1 ) - g_h( V_h; U_h, W_h ).
  \end{align*}
  
  Next, we introduce the $L^2(\Gamma_h(t))$ projection $\Lambda_h \colon L^2(\Gamma_h(t)) \to S_h(t)$. For $z \in L^2(\Gamma_h(t))$, we define $\Lambda_h z$ as the unique solution of
  \begin{equation}
    \label{eq:4}
    m_h( \Lambda_h z, \phi_h ) = m_h( z, \phi_h ) \quad \mbox{ for all } \phi_h \in S_h(t).
  \end{equation}
  For $z \in H^1(\Gamma_h(t))$, we will make use of the following bounds:
  \begin{equation}
    \label{eq:5}
    \norm{ \Lambda_h z }_{H^1(\Gamma_h(t))} \le c \norm{ z }_{H^1(\Gamma_h(t))}, \qquad
    \norm{ z - \Lambda_h z }_{L^2(\Gamma_h(t))} \le c h \norm{ z }_{H^1(\Gamma_h(t))}.
  \end{equation}
  These bounds follow since our triangulation is quasi-uniform.

  We first note that from our assumptions on $v$, we have $U_h (\nabla_{\Gamma} \cdot v)^{-\ell} \in H^1(\Gamma_h(t))$ with $\norm{ U_h (\nabla_{\Gamma} \cdot v)^{-\ell} }_{H^1(\Gamma_h(t))} \le c \norm{ U_h }_{H^1(\Gamma_h(t))}$. Next, we test \eqref{eq:ch-fem-u} with $\Lambda_h( U_h (\nabla_{\Gamma} \cdot v)^{-\ell} )$ and using \eqref{eq:5} and the Sobolev embedding (Lemma~\ref{lem:sobolev-embedh}), we see that
  \begin{align*}
    & \abs{ \int_{\Gamma_h(t)} W_h \Lambda_h \big( U_h  (\nabla_\Gamma \cdot v)^{-\ell} \big) \dd \sigma_h } \\
    & \le \eps \abs{ a_h ( U_h, \Lambda_h \big( U_h  (\nabla_\Gamma \cdot v)^{-\ell} \big) ) } + \frac{1}{\eps} \abs{ m_h ( \psi'(U_h), U_h (\nabla_\Gamma \cdot v )^\ell ) } \\
    & \qquad + \frac{1}{\eps} \abs{ m_h ( \psi'(U_h), \Lambda_h \big( U_h  (\nabla_\Gamma \cdot v)^{-\ell} \big) - U_h (\nabla_\Gamma \cdot v )^\ell ) } \\
    & \le c \left( \eps a_h( U_h, U_h ) + \frac{1}{\eps} m_h( \psi(U_h), 1 ) \right) + \frac{c h}{\eps} \norm{ U_h }_{H^1(\Gamma_h(t))}^4.
  \end{align*}
  Similarly, testing \eqref{eq:ch-fem-u} with $W_h$ leads to
  \begin{align*}
    m_h( W_h, W_h ) \le c \eps a_h( U_h, U_h ) + \frac12 a_h( W_h, W_h ) + \frac{c}{\eps^2} \norm{ U_h }_{H^1(\Gamma_h(t))}^6.
  \end{align*}

  Applying the geometric bound \eqref{eq:92}, the two previous bounds, a Poincar\'e inequality and the fact that the mass of $U_h$ is conserved, we infer that
  \begin{align*}
    & \abs{ g( V_h; U_h, W_h ) } \\
    & \quad \le \abs{ m_h( W_h, U_h \big( \nabla_{\Gamma_h} \cdot V_h - (\nabla_\Gamma \cdot v)^{-\ell} \big) ) }
    + \abs{ m_h \big( W_h, \Lambda_h( U_h ( \nabla_\Gamma \cdot v )^{-\ell} ) \big) } \\
    & \quad \le c m_h( U_h, U_h) + c h^2 m_h( W_h, W_h )
    + \abs{ m_h \big( W_h, \Lambda_h( U_h ( \nabla_\Gamma \cdot v )^{-\ell} ) \big) } \\
    & \quad \le c \left( \eps a_h( U_h, U_h ) + \frac{1}{\eps} m_h( \psi(U_h), 1 ) \right) + \frac{1}{2} a_h( W_h, W_h ) \\
    & \quad \qquad + c_\eps \left( h\, a_h( U_h, U_h )^2 + h^2\, a_h( U_h, U_h )^3 \right) + \tilde{C}_0(U_0),
  \end{align*}
  where $\tilde{C}_0(U_0)$ is a constant which only depends on the integral of $U_0$ on $\Gamma_{h,0}$. This leads to the estimate
  \begin{align*}
    & \dt \left( \eps a_h( U_h, U_h ) + \frac{1}{\eps} m_h( \psi(U_h), 1 ) \right) + a_h( W_h, W_h ) \\
    & \le c_\eps \left( \eps a_h( U_h, U_h ) + \frac{1}{\eps} m_h( \psi(U_h), 1 ) +  h\, a_h( U_h, U_h )^2 + h^2\, a_h( U_h, U_h )^3 \right) \\
    & \qquad + \tilde{C}_0(U_0).
  \end{align*}

  We will use the generalised Gronwall inequality from \eqref{eq:7} with $y_h = \eps a_h( U_h, U_h ) + \frac{1}{\eps} m_h( \psi(U_h), 1 )$, $z_h = a_h( W_h, W_h )$ and $\tilde{C} = \tilde{C}_0(U_0)$. Given $T$, there exists $h_1$, $0 < h_1 < h_0$, such that for $h < h_1$, we have
  \begin{equation*}
    1 - h (\E_0^h + \tilde{C}_0(U_0))^2 ( e^{ 2 ( 1 + h ) ct } - 1 ) > 0.
  \end{equation*}
  This gives the energy bound in \eqref{eq:72} with $C_0$ given by
  \begin{equation*}
    C_0 := \frac{ e^{(1+h)cT} (\E_0^h + \tilde{C}_0(U_0)) }{ \sqrt{ 1 - h (\E_0^h + \tilde{C}_0(U_0))^2 ( e^{ 2 ( 1 + h ) ct } - 1 ) }}
  \end{equation*}
  This implies, that if $h < h_1$, we have an energy bound on $(0,T)$ and hence can turn the short-time existence result in to existence over $(0,T)$ where $T$ is arbitrary. \qed
\end{proof}

\subsection{Geometric estimates}
\label{sec:geometric-estimates}

In this section, we will simply state the following geometric estimates without proof. Details can be found in \cite[Section~5]{DziEll13} except for \eqref{eq:ch-g-bound} and \eqref{eq:ch-b-bound} which can be found in \cite[Lemma~3.3.14]{Ran13}.

\begin{lemma}
  Let $\mu_h$ denote the quotient of surface measures $\mathrm{d} \sigma$ on $\Gamma(t)$ and $\mathrm{d} \sigma_h$ on $\Gamma_h(t)$ such that $\mu_h \dd \sigma_h = \mathrm{d} \sigma$; then
  \begin{subequations}
    \begin{align}
      \label{eq:87}
      \sup_{t \in [0,T]} \sup_{\Gamma_h(t)} \abs{ 1 - \mu_h } & \le c h^2 \\
      \label{eq:ch-mdh-muh-bound}
      \sup_{ t \in [0,T] } \norm{ \md_h \mu_h }_{L^\infty(\Gamma_h(t))} & \le c h^2.
    \end{align}
  \end{subequations}
\end{lemma}

\begin{lemma}
  \label{lem:ch-geometric-bounds}

  Let $Z_h, \phi_h \in S_h(t)$ with lifts $z_h, \varphi_h \in S_h^\ell(t)$. Then the following estimates hold for the given bilinear forms:
  \begin{subequations}
    \begin{align}
      \label{eq:ch-m-bound}
      \abs{ m_h( Z_h, \phi_h ) - m( z_h, \varphi_h ) }
      & \le c h^2 \norm{ Z_h }_{L^2(\Gamma_h(t))} \norm{ \phi_h }_{L^2(\Gamma_h(t))} \\
      \label{eq:ch-a-bound}
      \abs{ a_h( Z_h, \phi_h ) - a( z_h, \varphi_h ) }
      & \le c h^2 \norm{ \nabla_{\Gamma_h} Z_h }_{L^2(\Gamma_h(t))} \norm{ \nabla_{\Gamma_h} \phi_h }_{L^2(\Gamma_h(t))} \\
      \label{eq:ch-g-bound}
      \abs{ g_h( V_h; Z_h, \phi_h ) - g( v_h; z_h, \varphi_h ) }
      & \le c h^2 \norm{ Z_h }_{L^2(\Gamma_h(t))} \norm{ \phi_h }_{L^2(\Gamma_h(t))} \\
      \label{eq:ch-b-bound}
      \abs{ b_h( V_h; Z_h, \phi_h ) - b( v_h; z_h, \varphi_h ) }
      & \le c h^2 \norm{ \nabla_{\Gamma_h} Z_h }_{L^2(\Gamma_h(t))} \norm{ \nabla_{\Gamma_h} \phi_h }_{L^2(\Gamma_h(t))}.
    \end{align}
  \end{subequations}
\end{lemma}

Using the same reasoning, it is also clear that
\begin{equation}
  \label{eq:ch-mpsi-bound}
  \abs{ m_h( \psi'(Z_h), \phi_h) - m( \psi'(z_h), \varphi_h) } 
  \le c h^2 \norm{ \psi'(Z_h) }_{L^2(\Gamma_h(t))} \norm{ \phi_h }_{L^2(\Gamma_h(t))}.
\end{equation}

Similar results apply if the first argument is the material derivative of a finite element function:

\begin{lemma}
  For $Z_h\in S_h^T, \phi_h \in \tilde{S}_h^T$ with lifts $z_h, \varphi_h \in S_h^\ell(t)$ for each time, we have
  \begin{subequations}
    \begin{align}
      \label{eq:mmd-bound}
      \abs{ m_h( \md_h Z_h, \varphi_h ) - m( \md_h z_h, \phi_h ) }
      & \le c h^2 \norm{ \md_h Z_h }_{L^2(\Gamma_h(t))} \norm{ \phi_h }_{L^2(\Gamma_h(t))} \\
      \label{eq:amd-bound}
      \abs{ a_h( \md_h Z_h, \varphi_h ) - a( \md_h z_h, \phi_h ) }
      & \le c h^2 \norm{ \nabla_{\Gamma_h} (\md_h Z_h) }_{L^2(\Gamma_h(t))} 
      \norm{ \nabla_{\Gamma_h} \phi_h }_{L^2(\Gamma_h(t))}.
    \end{align}
  \end{subequations}
\end{lemma}

The next lemma bounds errors from the approximation of $v$ by $v_h$:

\begin{lemma}
  The difference between the continuous velocity $v$ and the discrete velocity $v_h$ on $\Gamma(t)$ can be estimated by
  \begin{equation}
    \label{eq:92}
    \abs{ v - v_h } + h \abs{ \nabla_\Gamma ( v - v_h ) } \le c h^2 \norm{ v }_{C^2(\N_T)} < c h^2.
  \end{equation}
\end{lemma}

This allows us to bound the error between the material derivatives on $\Gamma(t)$:

\begin{corollary}
  Suppose that $\eta\colon\G_T \to \R$ and $\md \eta$ and $\md_h \eta$ exist. For $\eta \in H^1(\Gamma(t))$, we have the estimate
  \begin{equation}
    \label{eq:93}
    \norm{ \md \eta - \md_h \eta }_{L^2(\Gamma(t))} 
    \le c h^2 \norm{ \nabla_\Gamma \eta }_{L^2(\Gamma(t))},
  \end{equation}
  and for $\eta \in H^2(\Gamma(t))$, we obtain
  \begin{equation}
    \label{eq:94}
    \norm{ \nabla_\Gamma ( \md \eta - \md_h \eta ) }_{L^2(\Gamma(t))}
    \le c h^2 \norm{ \eta }_{H^2(\Gamma(t))}.
  \end{equation}
\end{corollary}

\subsection{Ritz projection}
\label{sec:ritz-projection}

We conclude this section by constructing a discrete projection operator, similar to an interpolation operator. We define the Ritz projection operator, $\Pi_h z \in S_h(t)$, of $z \in H^1(\Gamma(t))$ as the unique solution of
\begin{equation}
  \label{eq:102}
  a_h ( \Pi_h z, \phi_h ) = a ( z, \varphi_h )
  \quad \mbox{ for all } \phi_h \in S_h(t), \mbox{ with lift } \varphi_h \in S_h^\ell(t)
\end{equation} 
and
\begin{equation*}
  \int_{\Gamma_h(t)} \Pi_h z \dd \sigma_h = \int_{\Gamma(t)} z \dd \sigma.
\end{equation*}
We will write $\pi_h z = (\Pi_h z)^\ell$ for the lift of the Ritz projection.

\begin{remark}
  This operator is the Ritz projection used by \cite{DuJuTia09}, but different to that used in other surface finite element analyses such as \cite{DziEll07a,DziEll13}, which use the operator $\Rh \colon H^2(\Gamma(t)) \to S_h^\ell(t)$ given as the unique solution of
  \begin{equation*}
    a ( \Rh z, \varphi_h ) = a ( z, \varphi_h ) \quad \mbox{ for all } \varphi_h \in S_h^\ell(t) 
    \quad \mbox{ and } \quad
    \int_{\Gamma(t)} \Rh z \dd \sigma = 0.
  \end{equation*}
\end{remark}

The following bounds are immediate:
\begin{theorem}
  \label{thm:Pih-bound}
  For $z \in H^1(\Gamma(t))$,
  \begin{equation}
    \label{eq:89}
    \norm{ \pi_h z }_{H^1(\Gamma(t))} \le c \norm{ z }_{H^1(\Gamma(t))}, \qquad \norm{ \pi_h z - z }_{L^2(\Gamma(t))} \le c h \norm{ z }_{H^1(\Gamma(t))}.
  \end{equation}
  For $z \in H^2(\Gamma(t))$, 
  \begin{equation}
    \label{eq:104}
    \norm{ \pi_h z - z }_{L^2(\Gamma(t))} +
    h \norm{ \nabla_\Gamma ( \pi_h z - z ) }_{L^2(\Gamma(t))}
    \le c h^2 \norm{ z }_{H^2(\Gamma(t))},
  \end{equation}
  and for $1 \le q \le \infty$, such that $H^1(\Gamma(t))$ embeds into $L^q(\Gamma(t))$,
  \begin{equation}
    \label{eq:2}
    \norm{ \nabla_\Gamma ( \pi_h z - z ) }_{L^q(\Gamma(t))} \le c h^{1 + \min( 0, n/q - n/2 )} \norm{ z }_{H^2(\Gamma(t))}.
  \end{equation}
\end{theorem}

\begin{proof}
  The $H^1$ stability result is clear and the $L^2$ error bound for a $H^1$ function follows from an Aubin-Nitsche trick. The $L^2$ results for $z \in H^2(\Gamma(t))$ follow from standard error estimates for the surface finite element \cite{Dzi88}. The $L^q$ result follows from the same splitting argument along with an inverse inequality. \qed
\end{proof}

\begin{corollary}
  The Ritz projection is bounded in $L^\infty$ and we have the bound
  \begin{equation}
    \label{eq:Pih-infty}
    \norm{ \Pi_h z }_{L^\infty(\Gamma_h(t))} \le \norm{ \pi_h z }_{L^\infty(\Gamma_h(t))} \le c \norm{ z }_{H^2(\Gamma(t))}.
  \end{equation}
\end{corollary}

\begin{proof}
  Let $1 < q < \infty$ be such that $H^1(\Gamma(t))$ embeds into $L^q(\Gamma(t))$ and such that $W^{1,q}(\Gamma(t))$ embeds into $L^\infty(\Gamma(t))$. The previous result, a Poincar\'e inequality and \eqref{eq:87} imply that
  \begin{align*}
    \norm{ \pi_h z - z }_{W^{1,q}(\Gamma(t))} & \le c \norm{ \nabla_\Gamma ( \pi_h z - z ) }_{L^q(\Gamma(t))} + \abs{ \int_{\Gamma(t)} \pi_h z - z \dd \sigma } \\
    & \le c \norm{ z }_{H^2(\Gamma(t))},
  \end{align*}
  for $h$ sufficiently small. We use a Sobolev embedding (\myref[Lemma]{lem:sobolev-embed}), to see
  \begin{equation*}
    \norm{ \pi_h z }_{L^\infty(\Gamma(t))} \le c \norm{ \pi_h z }_{W^{1,q}(\Gamma(t))} \le c \norm{ z }_{H^2(\Gamma(t))}.
  \end{equation*}
  It is clear that 
  \begin{equation*}
    \norm{ \Pi_h z }_{L^\infty(\Gamma_h(t))} = \norm{ \pi_h z }_{L^\infty(\Gamma(t))},
  \end{equation*}
  which completes the proof. \qed
\end{proof}

Since $\md_h \Pi_h z \neq \Pi_h \md_h z$, we also wish to have a bound on the discrete material derivative of this error for a function. We will assume that $z \in H^2(\Gamma(t))$ and $\md z \in H^2(\Gamma(t))$ for each $t$. Under this assumption, we may take a time derivative of \eqref{eq:102}, so that for all $\phi_h \in S_h^T$ with lift $\varphi_h \in S_h^{\ell,T}$,
\begin{equation}
  \label{eq:61}
  a_h( \md_h \Pi_h z, \phi_h )
  = a( \md_h z, \varphi_h ) + \big( b( v_h; z, \varphi_h ) - b_h( V_h; \Pi_h z, \phi_h ) \big).
\end{equation}
In fact using similar arguments to \myref[Lemma]{lem:pull-back}, we can construct a similar extension of a finite element function $\phi_h \in S_h(t)$ to a function $\tilde\phi_h \in S_h^T$ by
\begin{equation*}
  \tilde\phi_h( x, s ) = \sum_{j=1}^N \gamma_j \phi_j^N( x, s ) \quad \mbox{ for } ( x, s ) \in \G_{h,T} \quad \mbox{ where } 
  \phi_h( x ) = \sum_{j=1}^N \gamma_j \phi_j^N( x, t ).
\end{equation*}
Hence, we deduce that \eqref{eq:61} applies at each time $t \in (0,T)$ for $\phi_h \in S_h(t)$.

We start by proving two technical lemmas:

\begin{lemma}
  Given $z \colon \G_T \to \R$ with $z \in H^2(\Gamma(t))$ and $\md z \in H^2(\Gamma(t))$ for almost every time $t \in (0,T)$, then $\md_h \Pi_h z$ exists and we have the bound
  \begin{equation}
    \label{eq:ch-Pih-md-stability}
    \norm{ \nabla_{\Gamma_h} \big( \md_h \Pi_h z \big) }_{L^2(\Gamma_h(t))} \le c \big( \norm{ z }_{H^2(\Gamma(t))} + \norm{ \md z }_{H^2(\Gamma(t))} \big). 
  \end{equation}
\end{lemma}

\begin{proof}
  To show the bound, we start from \eqref{eq:61}, using a Young's inequality, \eqref{eq:104} and \eqref{eq:ch-b-bound} gives
  \begin{equation*}
    a_h( \md_h \Pi_h z, \phi_h ) \le c \big( \norm{ \nabla_\Gamma \md z }_{L^2(\Gamma(t))}^2 + \norm{ z }_{H^2(\Gamma(t))}^2 \big) + \frac{1}{2} \norm{ \nabla_{\Gamma_h} \phi_h }_{L^2(\Gamma_h(t))}^2.
  \end{equation*}
  Applying this bound with $\phi_h = \md_h \Pi_h z$ gives the estimate \eqref{eq:ch-Pih-md-stability}. \qed
\end{proof}

\begin{lemma}
  Define the function $T_h$ on $S_h^\ell(t)$ by
  \begin{equation}
    \label{eq:Th-defn}
    T_h( \varphi_h ) := a( \md_h ( \pi_h z - z ), \varphi_h ).
  \end{equation}
  Then we have the bound
  \begin{equation}
    \label{eq:T-bound}
    \abs{ T_h( \varphi_h ) } \le c h \left( \norm{ z }_{H^2(\Gamma(t))} + \norm{ \md z }_{H^2(\Gamma(t))} \right) \norm{ \nabla_{\Gamma} \varphi_h }_{L^2(\Gamma(t))}.
  \end{equation}
  Furthermore, for any $\eta \in H^2(\Gamma(t))$, we have that
  \begin{equation}
    \label{eq:T-bound2}
    \begin{aligned}
      \abs{ T_h( \varphi_h ) } 
      & \le c h \norm{ z }_{H^2(\Gamma(t))} \norm{ \nabla_\Gamma ( \varphi_h - \eta ) }_{L^2(\Gamma(t))} + c h^2 \norm{ z }_{H^2(\Gamma(t))} \norm{ \eta }_{H^2(\Gamma(t))} \\
      & \quad + c h^2 \left( \norm{ z }_{H^2(\Gamma(t))} + \norm{ \md z }_{H^2(\Gamma(t))} \right) \norm{ \nabla_\Gamma \varphi_h }_{L^2(\Gamma(t))}.
    \end{aligned}
  \end{equation}
\end{lemma}

\begin{proof}
  Using \eqref{eq:102} and \eqref{eq:61}, we see for $\phi_h \in S_h(t)$, with lift $\varphi_h \in S_h^\ell(t)$,
  \begin{equation*}
    \begin{aligned}
      T_h( \varphi_h ) & = a( \md_h \pi_h z, \varphi_h ) - a( \md_h z, \varphi_h ) \\
      & = b( v_h; z - \pi_h z, \varphi_h ) + \big( a( \md_h \pi_h z, \varphi_h ) - a_h( \md_h \Pi_h z, \phi_h ) \big) \\
      & \qquad + \big( b( v_h; \pi_h z, \varphi_h ) - b_h( V_h; \Pi_h z, \phi_h ) \big).
    \end{aligned}
  \end{equation*}
  Using our bound on the Ritz projection \eqref{eq:104}, and two geometric estimates \eqref{eq:ch-b-bound} and \eqref{eq:amd-bound}, we have that
  \begin{equation*}
    \begin{aligned}
      \abs{ T_h( \varphi_h ) }
      & \le c h \norm{ z }_{H^2(\Gamma(t))} \norm{ \nabla_\Gamma \varphi_h }_{L^2(\Gamma(t))} \\
      & \qquad + c h^2 \big( \norm{ \nabla_\Gamma \md_h \pi_h z }_{L^2(\Gamma(t))} + \norm{ \nabla_\Gamma \pi_h z }_{L^2(\Gamma(t))} \big) \norm{ \nabla_\Gamma \varphi_h }_{L^2(\Gamma(t))} \\
      & \le c h \left( \norm{ z }_{H^2(\Gamma(t))} + \norm{ \md z }_{H^2(\Gamma(t))} \right) \norm{ \nabla_\Gamma \varphi_h }_{L^2(\Gamma(t))}.
    \end{aligned}
  \end{equation*}
  
  We can improve this estimate by comparing $v_h$ to the smooth velocity $v$ and introducing a smooth function $\eta \in H^2(\Gamma(t))$. Then, we split the first term in $T_h(\varphi_h)$ into
  \begin{align*}
    & b( v_h; \pi_h z - z, \varphi_h ) \\
    & \quad = b( v_h - v; \pi_h z - z, \varphi_h ) + b( v; \pi_h z - z, \varphi_h - \eta ) + b( v; \pi_h z - z, \eta ).
  \end{align*}
  Using the smoothness of $\eta$, the final term, $b( v; \pi_h z - z, \eta )$, is bounded using an integration by parts argument given by \cite[p.~21]{DziEll13}:
  \begin{align*}
    b( v; \varphi, \eta ) & = \int_{\Gamma(t)} \sum_{i,j=1}^{n+1} H \nu_j \B(v)_{ij} \, \varphi \, \underline{D}_i \eta \dd \sigma
    - \int_{\Gamma(t)} \varphi \sum_{i,j=1}^{n+1} \underline{D}_j \big( \B(v)_{ij} \underline{D}_i \eta \big) \dd \sigma.
  \end{align*}
  Hence, we obtain
  \begin{equation*}
    \abs{ b( v; \varphi, \eta ) } \le c \norm{ \varphi }_{L^2(\Gamma(t))} \norm{ \eta }_{H^2(\Gamma(t))}.
  \end{equation*}
  Combining these calculations with \eqref{eq:92} and \eqref{eq:104}, we get
  \begin{equation*}
    \begin{aligned}
      \abs{ b( v_h; \pi_h z - z, \varphi_h ) }
      & \le c h \norm{ z }_{H^2(\Gamma(t))} \norm{ \nabla_\Gamma ( \varphi_h - \eta ) }_{L^2(\Gamma(t))} \\
      & \qquad + c h^2 \norm{ z }_{H^2(\Gamma(t))} \big( \norm{ \nabla_\Gamma \varphi_h }_{L^2(\Gamma(t))} + \norm{ \eta }_{H^2(\Gamma(t))} \big).
    \end{aligned}
  \end{equation*}
  Hence, we have
  \begin{equation*}
    \begin{aligned}
      \abs{ T_h( \varphi_h ) } 
      & \le c h \norm{ z }_{H^2(\Gamma(t)} \norm{ \nabla_\Gamma ( \varphi_h - \eta ) }_{L^2(\Gamma(t))} + c h^2 \norm{ z }_{H^2(\Gamma(t))} \norm{ \eta }_{H^2(\Gamma(t))} \\
      & \quad + c h^2 \left( \norm{ z }_{H^2(\Gamma(t))} + \norm{ \md z }_{H^2(\Gamma(t))} \right) \norm{ \nabla_\Gamma \varphi_h }_{L^2(\Gamma(t))},
    \end{aligned}
  \end{equation*}
  which is the second estimate. \qed
\end{proof}

These results allow us to show an estimate for the difference between the material derivative of a function and its Ritz projection.

\begin{lemma}
  \label{lem:md-Pih-bound}
  For $z \colon \G_T \to \R$ with $z, \md z \in H^2(\Gamma(t))$, we have
  \begin{equation}
    \label{eq:md-Pih-bound}
    \begin{aligned}
      & \norm{ \md_h( \pi_h z - z ) }_{L^2(\Gamma(t))} + h \norm{ \nabla_\Gamma \md_h (\pi_h z - z ) }_{L^2(\Gamma(t))} \\
      & \qquad\quad \le c h^2 \big( \norm{ z }_{H^2(\Gamma(t))} + \norm{ \md z }_{H^2(\Gamma(t))} \big).
    \end{aligned}
  \end{equation}
\end{lemma}

\begin{proof}
  We start by rewriting the error as
  \begin{equation}
    \label{eq:ch-mdpih-1}
    \begin{aligned}
      & a( \md_h ( \pi_h z - z ), \md_h ( \pi_h z - z ) ) \\
      & \quad = a( \md_h ( \pi_h z - z ), \md_h \pi_h z - I_h ( \md z ) ) + a( \md_h ( \pi_h z - z ), I_h( \md z ) - \md z ) \\
      & \quad\qquad + a( \md_h ( \pi_h z - z ), \md z - \md_h z ).
    \end{aligned}
  \end{equation}
  We can bound the first term on the right-hand side using \eqref{eq:T-bound} by
  \begin{align*}
    & \abs{ a( \md_h ( \pi_h z - z ), \md_h \pi_h z - I_h ( \md z ) ) } = \abs{ T_h( \md_h \pi_h z - I_h ( \md z ) ) } \\
    & \quad \le c h \big( \norm{ \md z }_{H^2(\Gamma(t))} + \norm{ z }_{H^2(\Gamma(t))} \big) \norm{ \nabla_\Gamma ( \md_h \pi_h z - I_h ( \md z ) ) }_{L^2(\Gamma(t))} \\
    & \quad \le c h \big( \norm{ \md z }_{H^2(\Gamma(t))} + \norm{ z }_{H^2(\Gamma(t))} \big) \norm{ \nabla_\Gamma \md_h ( \pi_h z - z ) }_{L^2(\Gamma(t))} \\
    & \quad\qquad + c h \norm{ \nabla_\Gamma \md_h ( \pi_h z - z ) }_{L^2(\Gamma(t))}^2.
  \end{align*}
  The second term is bounded using the approximation property \eqref{eq:ch-approx}:
  \begin{align*}
    \abs{ a( \md_h ( \pi_h z - z ), I_h( \md z ) - \md z ) } \le c h \norm{ \nabla_\Gamma \md_h ( \pi_h z - z ) }_{L^2(\Gamma(t))} \norm{ \md z }_{H^2(\Gamma(t))}.
  \end{align*}
  Finally, we use our estimate of the difference of material derivatives \eqref{eq:93} to bound the third term:
  \begin{align*}
    \abs{ a( \md_h ( \pi_h z - z ), \md z - \md_h z ) } \le c h^2 \norm{ \nabla_\Gamma \md_h ( \pi_h z - z ) }_{L^2(\Gamma(t))} \norm{ z }_{H^2(\Gamma(t))}.
  \end{align*}
  Combining these three bounds in \eqref{eq:ch-mdpih-1}, we get the desired gradient norm bound for $h$ sufficiently small.

  To show the $L^2$ bound, we use the Aubin-Nitsche trick. We start by writing $e = \md_h ( \pi_h z - z )$, then $e$ is in $L^2$ so can be set as the right-hand side for the dual problem: Find $\zeta \in H^1(\Gamma(t))$ such that
  \begin{equation}
    \label{eq:ch-dual-problem}
    a( \varphi, \zeta ) = m( e - c_0, \varphi ) \quad \mbox{ for all } \varphi \in H^1(\Gamma(t)) \mbox{, and } \int_{\Gamma(t)} \zeta \dd \sigma = 0,
  \end{equation}
  where $c_0 = \frac{1}{\abs{\Gamma(t)}} \int_{\Gamma(t)} e \dd \sigma$. We know \cite{Aub82} that \eqref{eq:ch-dual-problem} has a unique solution and satisfies the regularity result
  \begin{equation}
     \label{eq:ch-dual-reg}
     \norm{ \zeta }_{H^2(\Gamma(t))} \le c \norm{ e }_{L^2(\Gamma(t))}.
   \end{equation}
   We note that from $\int_{\Gamma_h(t)} \Pi_h z \dd \sigma_h = \int_{\Gamma(t)} z \dd \sigma$, that
   \begin{align*}
     \abs{ \Gamma(t) } \abs{ c_0 } & = \int_{\Gamma(t)} \md_h ( \pi_h z - z ) \dd \sigma \\
     & = \dt \int_{\Gamma(t)} \pi_h z - z \dd \sigma - \int_{\Gamma(t)} ( \pi_h z - z ) \nabla_\Gamma \cdot v_h \dd \sigma.
   \end{align*}
   We remark that from \eqref{eq:87} and \eqref{eq:ch-mdh-muh-bound}, we have
   \begin{align*}
     \dt \int_{\Gamma(t)} \pi_h z - z \dd \sigma 
     & = \dt \left( \int_{\Gamma(t)} \pi_h z \left( 1 - \frac{1}{\mu_h} \right) \dd \sigma \right) \\
     & \le c h^2 \big( \norm{ z }_{H^2(\Gamma(t))} + \norm{ \md z }_{H^2(\Gamma(t))} \big),
   \end{align*}
   and using \eqref{eq:104}, we infer
   \begin{equation*}
     \int_{\Gamma(t)} ( \pi_h z - z ) \nabla_\Gamma \cdot v_h \dd \sigma
     \le c \norm{ \pi_h z - z }_{L^2(\Gamma(t))} \le c h^2 \norm{ z }_{H^2(\Gamma(t))}.
   \end{equation*}
   This implies
   \begin{equation*}
     \abs{c_0} \le c h^2 \big( \norm{z}_{H^2(\Gamma(t))} + \norm{ \md z }_{H^2(\Gamma(t))} \big).
   \end{equation*}
   
   These calculations lead to
   \begin{equation}
     \label{eq:mdPi-bound}
     m( e, e ) - \abs{\Gamma(t)}^2 c_0^2 = a( \zeta, e ) = a( \zeta - I_h \zeta, e ) + T_h( I_h \zeta ).
   \end{equation}
   The first term on the right-hand side is bounded using the approximation property \eqref{eq:ch-approx} and the gradient norm bound on $e$, together with the dual regularity result \eqref{eq:ch-dual-reg}:
   \begin{align*}
     \abs{ a( \zeta - I_h, e ) }
     & \le c h^2 \norm{ e }_{L^2(\Gamma(t))} \big( \norm{ z }_{H^2(\Gamma(t))} + \norm{ \md z }_{H^2(\Gamma(t))} \big).
   \end{align*}
   The second term is estimated using the improved bound \eqref{eq:T-bound2} on $T_h( I_h \zeta )$ with $\eta = \zeta$. Applying the approximation \eqref{eq:ch-approx} we see
   \begin{align*}
     \abs{ T_h( I_h \zeta ) }
     & \le c h^2 \big( \norm{ z }_{H^2(\Gamma(t))} + \norm{ \md z }_{H^2(\Gamma(t))} \big) \norm{ e }_{L^2(\Gamma(t))}.
   \end{align*}
   Applying these two bounds in \eqref{eq:mdPi-bound} gives the desired result. \qed
\end{proof}

%--- well posedness ---------------------------------------------------------%

\section{Well-posedness of the continuous problem}
\label{sec:well-posedn-cont}

We use this section to show some properties of the continuous scheme based on the energy estimates coming from \myref[Theorem]{thm:ch-fem-wp} along with further some estimates. We will use these properties in later sections but they are also important results in their own right.

\subsection{Improved bounds on the finite element scheme}

In order to derive some improved bounds on $\md_h U_h$ and $W_h$, we will assume that $U_{h,0} = \Pi_h u_0$ with $u_0 \in H^2(\Gamma_0)$. It is clear that assumption \eqref{eq:ch-init-assumption} still holds in this case. In fact, we will make use of the bound
\begin{equation}
  \label{eq:smoothu0-bound}
  \E_0^h + \tilde{C}_0(U_0) \le c_\eps \big( 1 + \norm{ u_0 }_{H^2(\Gamma_0)}^2 + \norm{ u_0 }_{H^2(\Gamma_0)}^4 \big) + \tilde{\tilde{C}}_1(u_0) =: \tilde{C}_1(u_0),
\end{equation}
where
\begin{equation*}
  \tilde{\tilde{C}}_1(u_0) := c_\eps \left( \abs{ \int_{\Gamma_{0}} u_0 \dd \sigma } + \abs{ \int_{\Gamma_{0}} u_0 \dd \sigma }^2 + \abs{ \int_{\Gamma_{0}} u_0 \dd \sigma }^3 \right).
\end{equation*}
This implies the constant $C_0(U_0)$ from Theorem~\ref{thm:ch-fem-wp} can be bounded by
\begin{align*}
  C_0(U_0) & \le \max \left\{ \frac{ \exp(cT) \tilde{C}_1(u_0) }{ \sqrt{1 - \left( \exp( 2 c T ) - 1 \right) \left( h \tilde{C}_1(u_0) + h^2 \tilde{C}_1(u_0)^2 \right) }}, \tilde{C}_1(u_0)\right\} \\
  & =: C_1(u_0).
\end{align*}
This is not essential for well-posedness of the finite element method but will be used for the well-posedness results for the continuous problem.

First, we need a bound on $W_h|_{t=0}$:
\begin{lemma}
  \label{lem:ch-fem-w0-bound}
  Under the assumption that $u_0 \in H^2(\Gamma_0)$, the following bound holds for $W_h|_{t=0}$:
  \begin{equation}
    \label{eq:ch-fem-w0-bound}
    \norm{ W_h(0, \cdot) }_{L^2(\Gamma_h(0))} \le c_\eps \big( \norm{ u_0 }_{H^2(\Gamma_0)} + \norm{ u_0 }_{H^2(\Gamma_0)}^3 \big).
  \end{equation}
\end{lemma}

\begin{proof}
  Since $\alpha, \beta$ are $C^1([0,T];\R^N)$ in time (\myref[Theorem]{thm:ch-fem-wp}), we known that \eqref{eq:ch-fem-u} holds at time $t = 0$. We see that from the choice $U_{h,0} = \Pi_h u_0$, using Green's formula \eqref{eq:greens-form}, we have
  \begin{equation*}
    a_h( U_{h,0}, W_h(0,\cdot) ) = a( u_0, w_h( 0, \cdot) ) = - m( \Delta_\Gamma u_0, w_h( 0, \cdot ) ).
  \end{equation*}
  This implies that
  \begin{align*}
    m_h( W_h( 0, \cdot ), W_h( 0, \cdot ) )
    & = \eps a_h( U_{h,0}, W_h( 0, \cdot ) ) + \frac1\eps m_h( \psi'( U_{h,0} ), W_h( 0, \cdot ) ) \\
    & \le c_\eps \big( \norm{ u_0 }_{H^2(\Gamma_0)} + \norm{ u_0 }_{H^2(\Gamma_0)}^3 \big) \norm{ W_h( 0, \cdot ) }_{L^2(\Gamma_h(0))}.
  \end{align*}
  In the last line we have used \eqref{eq:79} and the Sobolev embedding of $H^1(\Gamma(t)) \hookrightarrow L^6(\Gamma(t))$ (\myref[Lemma]{lem:sobolev-embed}). \qed
\end{proof}

From \myref[Theorem]{thm:ch-fem-wp}, we see that $\beta \in C^1([0,T], \R^N)$ so $\md_h W_h$ exists. Hence, we may take the time derivative of \eqref{eq:ch-fem-u} to see, for $\phi_h \in S_h^T$,
\begin{equation}
  \label{eq:ch-fem-u-dt}
  \begin{aligned}
    & \eps \big( a_h( \md_h U_h, \phi_h ) + b_h( V_h; U_h, \phi_h ) \big) \\
    & \quad + \frac1\eps \big( m_h( \psi''(U_h) \md_h U_h, \phi_h ) + g_h( V_h; \psi'(U_h), \phi_h ) \big) \\
    & \quad - \big( m_h( \md_h W_h, \phi_h ) + g_h( V_h; W_h, \phi_h ) \big) = 0.
  \end{aligned}
\end{equation}

\begin{lemma}
  \label{lem:ch-fem-bound-improved}
  Under the assumption that $u_0 \in H^2(\Gamma_0)$, we have the bound
  \begin{equation}
    \begin{aligned}
      & \eps \int_0^T \norm{ \md_h U_h }_{L^2(\Gamma_h(t))}^2 \dd t
      + \sup_{t \in (0,T)} \norm{ W_h }_{L^2(\Gamma_h(t))}^2  \le C_2(u_0).
    \end{aligned}
  \end{equation}
  with
  $C_2(u_0)$ given by
  \begin{equation*}
    C_2(u_0) := c_\eps \big( \norm{ u_0 }_{H^2(\Gamma_0)} + \norm{ u_0 }_{H^2(\Gamma_0)}^3 + C_1(u_0) + C_1(u_0)^2  \big).
  \end{equation*}
\end{lemma}

\begin{proof}
  We start by subtracting \eqref{eq:ch-fem-u-dt} tested with $W_h$ from \eqref{eq:ch-fem-w'} tested with $\eps \md_h U_h$ and use the transport formula \eqref{eq:mh-transport} to arrive at
  \begin{equation}
    \label{eq:ch-imp-bound-2}
    \begin{aligned}
      & \eps m_h( \md_h U_h, \md_h U_h ) + \frac12 \dt m_h( W_h, W_h ) \\
      & \quad = - \eps \big( g_h( V_h; \md_h U_h, U_h ) + b_h( V_h; U_h, W_h ) \big) \\
      & \quad\qquad + \frac1\eps \big( m_h( \psi''(U_h) \md_h U_h, W_h ) + g_h( V_h; \psi'(U_h), W_h ) \big) - \frac12 g_h( V_h; W_h, W_h ).
    \end{aligned}
  \end{equation}
  Note that using a H\"{o}lder inequality, Young's inequality with $\eps$, and the Sobolev embedding (\myref[Lemma]{lem:sobolev-embedh}) we have
  \begin{align*}
    & \abs{ m_h( \psi''(U_h) \md_h U_h, W_h ) } \\
    & \quad = \abs{ \int_{\Gamma_h(t)} \psi''( U_h ) \md_h U_h W_h \dd \sigma_h }
    = \abs{ \int_{\Gamma_h(t)} ( 3 U_h^2 - 1 ) \md_h U_h W_h \dd \sigma_h } \\
    & \quad \le \frac{\eps}{4} \norm{ \md_h U_h }_{L^2(\Gamma_h(t))}^2 
    + c_\eps \big( \norm{ U_h }_{H^1(\Gamma_h(t))}^2 \norm{ W_h }_{H^1(\Gamma_h(t))}^2 + \norm{ W_h }_{L^2(\Gamma_h(t))}^2 \big).
  \end{align*}
  Applying this estimate in \eqref{eq:ch-imp-bound-2}, we have
  \begin{align*}
    & \eps m_h( \md_h U_h, \md_h U_h ) + \dt m_h( W_h, W_h ) \\
    & \quad \le c_\eps \left( \norm{ U_h }_{H^1(\Gamma_h(t))}^2
    + \norm{ W_h }_{H^1(\Gamma_h(t))}^2 
    + \norm{ U_h }_{H^1(\Gamma_h(t))}^2 \norm{ W_h }_{H^1(\Gamma_h(t))}^2 \right).
  \end{align*}
  Integrating in time using a Gronwall inequality gives us
  \begin{align*}
    & \eps \int_0^T \norm{ \md_h U_h }_{L^2(\Gamma_h(t))}^2 \dd t
    + \sup_{t \in (0,T)} \norm{ W_h }_{L^2(\Gamma_h(t))}^2 \\
    & \quad \le \norm{ W_h( \cdot, 0 ) }_{L^2(\Gamma_h(t))}^2 + c_\eps \int_0^T \left( \norm{ U_h }_{H^1(\Gamma_h(t))}^2 + \norm{ W_h }_{H^1(\Gamma_h(t))}^2  \right) \dd t \\
    & \quad\qquad + c_\eps \sup_{t \in (0,T)} \norm{ U_h }_{H^1(\Gamma_h(t))}^2 \int_0^T \norm{ W_h }_{H^1(\Gamma_h(t))}^2 \dd t.
  \end{align*}
  Applying the bounds from \myref[Theorem]{thm:ch-fem-wp}, \myref[Lemma]{lem:ch-fem-w0-bound} and \eqref{eq:smoothu0-bound} completes the proof. \qed
\end{proof}

\subsection{Existence}

The idea of the existence proof is to show that the lift of the solutions to finite element scheme \eqref{eq:ch-fem} converges, along a subsequence, to a solution of the continuous equations.

We suppose that $u_0 \in H^2(\Gamma_0)$ is a given function. In this section, we will take $U_{h,0} = \Pi_h u_0$ with $\Pi_h$ the Ritz projection defined in \eqref{eq:102}. Since the Ritz projection is stable in $H^1$, the stability bound in \myref[Theorem]{thm:ch-fem-wp} holds independently of $h$. Furthermore, the stability bounds from \myref[Lemma]{lem:ch-norm-equiv} imply we may transform this bound to $\{ \Gamma(t) \}$ and bound the lifts $u_h = U_h^\ell$ and $w_h = W_h^\ell$ by
\begin{align*}
  & \sup_{t \in (0,T)} \int_{\Gamma(t)} \frac{\eps}{2} \abs{ \nabla_{\Gamma} u_h }^2 + \frac1\eps \psi( u_h ) \dd \sigma + \int_0^T \norm{ w_h }_{H^1(\Gamma(t))}^2 \dd t \le C_1(u_0).
\end{align*}
Our assumption that $u_0 \in H^2(\Gamma_0)$ allows the use of the improved bounds in \myref[Lemma]{lem:ch-fem-bound-improved}. Using similar lifting arguments we have
\begin{equation*}
  \int_0^T \norm{ \md u_h }_{L^2(\Gamma(t))}^2 \dd t \le C_2(u_0).
\end{equation*}
These bounds, along with the conservation of mass property \eqref{eq:mass-conserv}, imply that $u_h$ is uniformly bounded in $L^\infty_{H^1} \cap H^1(\G_T)$ and $w_h$ in $L^2_{H^1}$. Hence, we may extract subsequences (for which we will still use the subscript $h$), and functions $\bar{u}$ and $\bar{w}$ with $\bar{u} \in L^\infty_{H^1} \cap H^1(\G_T)$, and $\bar{w} \in L^2_{H^1}$ such that
\begin{equation}
  \label{eq:ch-weak-conv1}
  \begin{aligned}
    u_h \weakto \bar{u} & \quad \mbox{ weakly in } H^1(\G_T) \qquad
    w_h \weakto \bar{w} & \quad \mbox{ weakly in } L^2_{H^1}.
  \end{aligned}
\end{equation}
We remark that these results imply $\md \bar{u} \in L^2_{L^2}$ and $\md u_h \weakto \md \bar{u}$ weakly in $L^2_{L^2}$. Furthermore, from the compactness result (\myref[Proposition]{prop:compact-embed}) we infer that we may take a further subsequence (still denoted $u_h$) such that
\begin{equation*}
%  \label{eq:ch-ae-conv}
  u_h \to \bar{u} \quad \mbox{ almost everywhere in } \G_T.
\end{equation*}
Using a Dominated Convergence Theorem-type argument \cite[Lemma~8.3]{Rob01}, since $\norm{ \psi'(u_h) }_{L^2(\G_T)} \le c \norm{ u_h }_{H^1(\G_T)}^3$ is bounded independently of $h$, we infer that
\begin{equation}
  \label{eq:ch-psi-conv}
  \psi'(u_h) \weakto \psi'(\bar{u}) \quad \mbox{ weakly in } L^2_{L^2}.
\end{equation}

We will show that $\bar{u}$ and $\bar{w}$ satisfy \eqref{eq:ch-weak}. For $\varphi \in L^2_{H^1}$, we write $\phi_h = \Pi_h \varphi$, where $\Pi_h$ is the Ritz-projection \eqref{eq:102}, and $\varphi_h = \phi_h^\ell = \pi_h \varphi$. 

Using \eqref{eq:ch-fem}, we have
\begin{equation}
  \label{eq:ch-ex-1}
  \begin{aligned}
    & m( \md \bar{u}, \varphi ) + g( v; \bar{u}, \varphi ) + a( w, \varphi ) \\
    & \quad = \big( m( \md \bar{u}, \varphi ) - m_h( \md_h U_h, \phi_h ) \big)
    + \big( g( v; \bar{u}, \varphi ) - g_h( V_h; U_h, \phi_h ) \big) \\
    & \quad\qquad + \big( a( w, \varphi ) - a_h( W_h, \phi_h ) \big).
  \end{aligned}
\end{equation}
and
\begin{equation}
  \label{eq:ch-ex-2}
  \begin{aligned}
    & \eps a( \bar{u}, \varphi ) + \frac1\eps m( \psi'( \bar{u} ), \varphi ) - m( \bar{w}, \varphi ) \\
    & \quad = \eps \big( a( \bar{u}, \varphi ) - a_h( U_h, \phi_h ) \big)
    + \frac1\eps \big( m( \psi'( \bar{u} ), \varphi ) - m_h( \psi'(U_h), \phi_h ) \big) \\
    & \quad\qquad - \big( m( \bar{w}, \varphi ) - m_h( W_h, \phi_h ) \big).
  \end{aligned}
\end{equation}

We may use the geometric estimates shown in Section~\ref{sec:geometric-estimates} and the bounds on the Ritz projection from Section~\ref{sec:ritz-projection} to see to bound the terms on the right-hand sides of these equations. We will denote by $c(h)$ a generic constant depending on $h$, which may also depend on $\eps$, such that $c(h) \to 0$ as $h \to 0$. Integrating in time, this implies
\begin{align*}
  & \int_0^T m( \md \bar{u}, \varphi ) + g( v; \bar{u}, \varphi ) + a( \bar{w}, \varphi ) \dd t \\
  & \quad \le \int_0^T m( \md \bar{u} - \md u_h, \varphi ) + g( v; \bar{u} - u_h, \varphi ) + a( \bar{w} - w_h, \varphi ) \dd t \\
  & \quad\quad + c(h) \int_0^T \big( \norm{ \md_h u_h }_{L^2(\Gamma(t))} + \norm{ u_h }_{H^1(\Gamma(t))} + \norm{ \nabla_\Gamma w_h }_{L^2(\Gamma(t))} \big) \norm{ \varphi }_{H^1(\Gamma(t))} \dd t,
\end{align*}
and
\begin{align*}
  & \int_0^T \eps a( \bar{u}, \varphi ) + \frac1\eps m( \psi'(\bar{u}), \varphi ) - m( \bar{w}, \varphi ) \dd t \\
  & \quad \le \int_0^T \eps a( \bar{u} - u_h, \varphi ) + \frac1\eps m( \psi'(\bar{u}) - \psi'(u_h), \varphi ) - m( \bar{w} - w_h, \varphi ) \dd t \\
  & \quad\qquad + c(h) \int_0^T \big( \norm{ u_h }_{H^1(\Gamma(t))} + \norm{ w_h }_{L^2(\Gamma(t))} \big) \norm{ \varphi }_{H^1(\Gamma(t))} \dd t.
\end{align*}
We may send $h \to 0$ in the right-hand sides of both previous equations, and use the convergence results \eqref{eq:ch-weak-conv1} and \eqref{eq:ch-psi-conv}, so that for all $\varphi \in L^2_{H^1}$ we arrive at
\begin{align*}
  \int_0^T m( \md \bar{u}, \varphi ) + g( v; \bar{u}, \varphi ) + a( \bar{w}, \varphi ) \dd t & = 0 \\
  \int_0^T \eps a( \bar{u}, \varphi ) + \frac1\eps m( \psi'(u), \varphi ) - m( w, \varphi ) & = 0.
\end{align*}
Finally, we use \myref[Lemma]{lem:tiredy} to transform this equality into a almost everywhere in time equality so that the pair $\bar{u}, \bar{w}$ satisfy \eqref{eq:ch-weak}.

To show that $\bar{u}$ achieves the initial condition, we start by choosing $\varphi \in C^2( \bar\G_T )$ and continue with the notation $\varphi_h = \pi_h \varphi$. Using the discrete transport formula \eqref{eq:83}, the lift of the finite element solution $u_h$ satisfies
\begin{equation*}
  \int_0^T m( u_h, \varphi_h ) \dot\alpha \dd t 
  = - \int_0^T \big( m( \md_h u_h, \varphi_h ) + m( u_h, \md_h \varphi_h ) + g( v_h; u_h, \varphi_h ) \big) \alpha \dd t,
\end{equation*}
for all $\alpha \in C_c^\infty(0,T)$. Using similar limiting arguments as above, with the addition of \eqref{eq:md-Pih-bound}, we obtain the identity
\begin{equation*}
  \int_0^T m( \bar{u}, \varphi ) \dot \alpha \dd t
  = - \int_0^T \big( m( \md \bar{u}, \varphi ) + m( \bar{u}, \md \varphi ) + g( v; \bar{u}, \varphi ) \big) \alpha \dd t.
\end{equation*}
In fact, by density of $C^2(\G_T)$ functions in $H^1(\G_T)$ \cite[Theorem~2.4]{Heb00}, we see that this equality holds for all $\varphi \in H^1(\G_T)$. This implies that $m( \bar{u}, \varphi )$ is weakly differentiable as a function on $(0,T)$ with weak derivative $m( \md \bar{u}, \varphi ) + m( \bar{u}, \md \varphi ) + g( v; \bar{u}, \varphi )$. Since $\bar{u}, \varphi \in H^1(\G_T)$, this weak derivative is a function in $L^1(0,T)$, and hence we infer that $m( \bar{u}, \varphi )$ is absolutely continuous on $[0,T]$ \cite[Section~4.9, Theorem~1]{EvaGar92}. In particular, $\norm{ \bar{u} }_{L^2(\Gamma(t))}$ is absolutely continuous, which means that we can interpret $\bar{u}(\cdot,0)$ as an $L^2(\Gamma_0)$ function. The absolute continuity of $m(\bar{u},\varphi)$ for $\varphi \in C^2(\G_T)$ also implies that
\begin{equation}
  \label{eq:clever}
  \begin{aligned}
    & m\big( \bar{u}(\cdot, t), \varphi(\cdot, t) \big)
    - m\big( \bar{u}(\cdot, 0), \varphi(\cdot, 0) \big) \\
    & \quad\qquad = \int_0^t m( \md \bar{u}, \varphi ) + g( v; \bar{u}, \varphi ) + m( \bar{u}, \md \varphi ) \dd s.
  \end{aligned}
\end{equation}

Next, we choose $\varphi \in C^2(\G_T)$ with $\varphi( \cdot, T ) = 0$. It is clear that $\varphi \in L^2_{H^1}$, hence we can use the limiting equation and \eqref{eq:clever} to see that
\begin{equation*}
  \int_0^T - m( \bar{u}, \md \varphi ) + a( \bar{w}, \varphi ) \dd t = m( \bar{u}(\cdot,0), \varphi(\cdot,0)).
\end{equation*}
We can do the same in the finite element scheme for $\phi_h = \Pi_h \varphi$, using the transport formula \eqref{eq:mh-transport}:
\begin{equation*}
  \int_0^T - m_h( U_h, \md_h \phi_h ) + a_h( W_h, \phi_h ) \dd t = m_h( \Pi_h u_0, \phi_h( \cdot,0 )).
\end{equation*}
The above calculations show that we are able to take the limit $h \to 0$ (in the appropriate sense) to see that
\begin{equation*}
  \int_0^T - m( \bar{u}, \md \varphi ) + a( \bar{w}, \varphi ) \dd t = m( u_0, \varphi( \cdot, 0 ) ).
\end{equation*}
Therefore, by comparing terms, we have shown that $\bar{u}(\cdot,0) = u_0$ almost everywhere in $\Gamma_0$ by the Fundamental Lemma of the Calculus of Variations.

Hence we have shown the following result:

\begin{theorem}
  \label{thm:ch-wp}
  Given $u_0 \in H^2(\Gamma_0)$ there exists a weak solution pair $(u,w)$ of the Cahn-Hilliard equation in the sense of \myref[Definition]{def:ch-solution}. Furthermore the solution satisfies the energy bound
  \begin{equation}
    \label{eq:ch-wp-bound}
    \begin{aligned}
      & \sup_{t \in (0,T)} \int_{\Gamma(t))} \frac{\eps}{2} \abs{ \nabla_\Gamma u}^2 + \frac1\eps \psi(u) \dd \sigma + \int_0^T \norm{ w }_{H^1(\Gamma(t))}^2 \dd t \le C_2(u_0).
    \end{aligned}
  \end{equation}
\end{theorem}

\subsection{Uniqueness}

To show the uniqueness result, we require an inverse Laplacian on $\Gamma(t)$. For $z \in L^2(\Gamma(t))$ with $\int_{\Gamma(t)} z \dd \sigma = 0$, we define $\G z$ the inverse Laplacian of $z$ as the unique solution of
\begin{equation}
  \label{eq:ch-inv-lap}
  a( \G z, \varphi ) = m( z, \varphi ) \quad \mbox{ for all } \varphi \in H^1(\Gamma(t)) \mbox{, and } \int_{\Gamma(t)} \G z \dd \sigma = 0.
\end{equation}

We will write 
\begin{equation*}
  \norm{ z }_{-1} := \norm{ \nabla_\Gamma \G z }_{L^2(\Gamma(t))} = a( \G z, \G z )^{\frac12}.
\end{equation*}
and remark that
\begin{equation*}
  \norm{ z }_{-1}^2 = m( \G z, z ).
\end{equation*}

It is clear that if $z \in L^2(\Gamma(t))$ then $\G z \in H^1(\Gamma(t))$. We also have a similar result for the material derivative of $\G z$.

\begin{lemma}
  If $z \in H^1(\G_T)$, with $\int_{\Gamma(t)} z \dd \sigma = 0$, then $\G z \in H^1(\G_T)$.
\end{lemma}

\begin{proof}
  It is clear that $\G z \in L^2_{H^1}$ for $z \in L^2_{H^1}$. It is left to show $\md \G z \in L^2_{L^2}$. We start by taking a time derivative of \eqref{eq:ch-inv-lap} so that for $\xi \in H^1(\G_T)$:
  \begin{equation*}
    a( \md \G z, \xi ) + a( \G z, \md \xi ) + b( v; \G z, \xi )
    = m( \md z, \xi ) + m( z, \md \xi ) + g( v; z, \xi ).
  \end{equation*}
  From \myref[Lemma]{lem:pull-back}, given $\varphi \in H^1(\Gamma(t^*))$, we can construct $\tilde\varphi \colon \G_T \to \R$, with $\tilde\varphi \in H^1(\Gamma(t))$ for all $t \in [0,T]$ and $\md \tilde\varphi = 0$. Thus, we have that
  \begin{equation*}
    a( \md \G z, \tilde\varphi ) + b( v; \G z, \tilde\varphi )
    = m( \md z, \tilde\varphi ) + g( v; z, \tilde\varphi ) \quad \mbox{ for } t \in (0,T),
  \end{equation*}
  and, in particular, at $t = t^*$,
  \begin{equation*}
    a( \md \G z, \varphi ) + b( v; \G z, \varphi )
    = m( \md z, \varphi ) + g( v; z, \varphi ).
  \end{equation*}
  Also, we have that
  \begin{equation*}
    m( \md z, 1 ) + g( v; z, 1 ) - b( v; \G z, 1 ) = \dt \int_{\Gamma(t)} z \dd \sigma = 0.
  \end{equation*}
  These calculations imply that $\md \G z$ solves the elliptic problem:
  \begin{equation*}
    a( \md \G z, \varphi ) = m( \md z, \varphi ) + g( v; z, \varphi ) - b( v; \G z, \varphi ) \quad \mbox{ for all } \varphi \in H^1(\Gamma(t^*)).
  \end{equation*}
  This implies that $\md \G z \in H^1(\Gamma(t^*))$ with the bound
  \begin{equation*}
    \norm{ \md \G z }_{H^1(\Gamma(t^*))} \le c \left( \norm{ \md z }_{L^2(\Gamma(t^*))} + \norm{ z }_{L^2(\Gamma(t^*))} + \norm{ z }_{-1} \right).
  \end{equation*}
  Integrating in time gives the desired result. \qed
\end{proof}

\begin{theorem}
  There is at most one solution to \eqref{eq:ch-weak}.
\end{theorem}

\begin{proof}
  We suppose that $(u_1, w_1)$ and $(u_2, w_2)$ are solutions to \eqref{eq:ch-weak}. We will write $\eta^u = u_1 - u_2$ and $\eta^w = w_1 - w_2$. For $\varphi \in L^2_{H^1}$, we know that
  \begin{subequations}
    \begin{align}
      \label{eq:ch-unique-u}
      m( \md \eta^u, \varphi ) + g( v; \eta^u, \varphi ) + a( \eta^w, \varphi ) & = 0 \\
      \label{eq:ch-unique-w}
      \eps a( \eta^u, \varphi ) + \frac1\eps( \psi'(u_1) - \psi'(u_2), \varphi ) - m( \eta^w, \varphi ) & = 0.
    \end{align}
  \end{subequations}
  Testing \eqref{eq:ch-unique-u} with $\varphi = 1$ tells us that
  \begin{equation*}
    \int_{\Gamma(t)} \eta^u \dd \sigma = \int_{\Gamma_0} \eta^u \dd \sigma = 0,
  \end{equation*}
  Hence, since $\G \eta^u$ is well defined and $\G \eta^u \in H^1(\G_T)$, we may test the first equation with $\G \eta^u$, and apply \eqref{eq:96}, to obtain
  \begin{equation}
    \label{eq:ch-unique-2}
    \dt \norm{ \eta^u }_{-1}^2 + m( \eta^w, \eta^u ) = m( \eta^u, \md \G \eta^u ).
  \end{equation}
  Next, using the monotonicity of $z \mapsto z^3$, testing the second equation with $\eta^u$ gives
  \begin{equation}
    \label{eq:ch-unique-1}
    \eps a ( \eta^u, \eta^u ) - \frac1\eps m( \eta^u, \eta^u ) \le m( \eta^w, \eta^u ).
  \end{equation}
  Taking the sum of \eqref{eq:ch-unique-1} and \eqref{eq:ch-unique-2}, we obtain
  \begin{equation*}
    \dt \norm{ \eta^u }_{-1}^2 + \eps \norm{ \nabla_\Gamma \eta^u }_{L^2(\Gamma(t))}^2
    \le \frac1\eps m( \eta^u, \eta^u ) + m( \eta^u, \md \G \eta^u ).
  \end{equation*}
  For the first term on the right-hand side, we see that
  \begin{equation*}
    \frac{1}{\eps} m( \eta^u, \eta^u ) = \frac{1}{\eps} a( \eta^u, \G \eta^u )
    \le \frac{\eps}{2} \norm{ \nabla_\Gamma \eta^u }_{L^2(\Gamma(t))}^2 + c_\eps \norm{ \eta^u }_{-1}^2,
  \end{equation*}
  and for the second, we have
  \begin{align*}
    m( \eta^u, \md \G \eta^u ) = a( \G \eta^u, \md \G \eta^u )
    & = \frac{1}{2} \dt a( \G \eta^u, \G \eta^u ) - \frac{1}{2} b( v; \G \eta^u, \G \eta^u ) \\
    & \le \frac{1}{2} \dt \norm{ \eta^u }_{-1}^2 + c \norm{ \eta^u }_{-1}^2.
  \end{align*}
  Combining these terms, we obtain the estimate
  \begin{equation*}
    \dt \norm{ \eta^u }_{-1}^2 + \eps \norm{ \nabla_\Gamma \eta^u }_{L^2(\Gamma(t))}^2
    \le c_\eps \norm{ \eta^u }_{-1}^2.
  \end{equation*}
  We next use a Gronwall inequality and integration in time to see
  \begin{equation*}
    \sup_{t \in (0,T)} \norm{ \eta^u }_{-1}^2 
    + \eps \int_0^T \norm{ \nabla_\Gamma \eta^u }_{L^2(\Gamma(t))}^2
    \le c_\eps \norm{ \eta^u |_{t=0} }_{-1}^2 = 0.
  \end{equation*}
  Since $\int_{\Gamma(t)} \eta^u \dd \sigma = 0$, we apply a Poincar\'e inequality to arrive at
  \begin{equation*}
    \int_0^T \norm{ \eta^u }_{L^2(\Gamma(t))}^2 \dd t \le \int_0^T \norm{ \nabla_\Gamma \eta^u }_{L^2(\Gamma(t))}^2 \dd t = 0.
  \end{equation*}
  This shows that $u_1 = u_2$.
  
  Now, we know that $\eta^u = 0$ and thus testing \eqref{eq:ch-unique-u} with $\eta^w$ gives
  \begin{equation*}
    m( \eta^w, \eta^w ) = \eps a( \eta^u, \eta^w ) + \frac{1}{\eps} m( \psi'(u_1) - \psi'(u_2), \eta^w ) = 0.
  \end{equation*}
  This shows that $w_1 = w_2$. \qed
\end{proof}

\subsection{Regularity}
\label{sec:ch-reg}

In this section, we show that the solution enjoys $H^2$ regularity.

\begin{theorem}
  [Regularity]
  Let $u_0 \in H^2(\Gamma_0)$ and $(u,w)$ be the solution pair of \eqref{eq:ch-weak}, then $u \in L^\infty_{H^2}$ and $w \in L^2_{H^2}$, with the bounds
  \begin{equation}
    \eps \sup_{t \in (0,T)} \norm{ u }_{H^2(\Gamma(t))}^2 + \int_0^T \norm{ w }_{H^2(\Gamma(t))}^2 \le C_2(u_0).
  \end{equation}
\end{theorem}

\begin{proof}
  Using the improved estimates from \myref[Lemma]{lem:ch-fem-bound-improved}, we have that
  \begin{equation}
    \label{eq:ch-reg-1}
    \begin{aligned}
      & \eps \int_0^T \norm{ \md u }_{L^2(\Gamma(t))}^2 + \sup_{t \in (0,T)} \norm{ w }_{L^2(\Gamma(t))}^2 \le C_2(u_0).
    \end{aligned}
  \end{equation}
  Now, we can translate the fact that $(u,w)$ are solutions of \eqref{eq:ch-weak} into
  \begin{align*}
    \eps a( u, \varphi ) & = m( f_1, \varphi ) \\
    a( w, \varphi ) & = m( f_2, \varphi ) \quad \mbox{ for all } \varphi \in H^1(\Gamma(t)),
  \end{align*}
  for $f_1 = w - \frac{1}{\eps} \psi'(u)$ and $f_2 = \md u + u \nabla_{\Gamma} \cdot v$. Notice that 
  \begin{equation*}
    \int_{\Gamma(t)} f_1 \dd \sigma = \int_{\Gamma(t)} f_2 \dd \sigma = 0.
  \end{equation*}
  The above improved bounds combined with the bounds in \myref[Theorem]{thm:ch-wp} gives $f_1 \in L^\infty_{L^2}$ and $f_2 \in L^2_{L^2}$. Standard theory of elliptic partial differential equations \cite{Aub82} gives $u \in L^\infty_{H^2}$ and $w \in L^2_{H^2}$. The proof is completed by using the bounds in \eqref{eq:ch-wp-bound} and \eqref{eq:ch-reg-1} on $f_1$ and $f_2$. \qed
\end{proof}

%--- analysis of the finite element scheme ----------------------------------%

\section{Error analysis of finite element scheme}

In this section, we show an error bound for the surface finite element method described in \myref[Section]{sec:finite-elem-appr}. The proof relies on decomposing the errors into errors between the smooth solution and Ritz projection and between the Ritz projection and discrete solution. In contrast to previous studies of partial differential equations on surfaces \cite{Dzi88,DziEll07,DziEll13}, we show an error bound on $\Gamma_h(t)$ instead of $\Gamma(t)$. This allows an easier treatment of the non-linear terms.

We will assume that $u_0, u$ and $w$ are bounded in the following norms
\begin{equation}
  \label{eq:ch-error-assump}
  \norm{u_0}_{H^2(\Gamma_0)}^2 + \sup_{t \in (0,T)} \norm{ u }_{H^2(\Gamma(t))}^2 + \int_0^T \norm{ w }_{H^2(\Gamma(t))}^2 + \norm{ \md u }_{H^2(\Gamma(t))}^2 \dd t < + \infty.
\end{equation}
\myref[Section]{sec:ch-reg} shows how to bound some of these terms. Again, we will assume that the initial condition of the finite element scheme is given by the Ritz projection:
\begin{equation}
  \label{eq:115}
  U_{h,0} = \Pi_h u_0.
\end{equation}

The error bound we will show is stated as follows:

\begin{theorem}
  \label{thm:ch-error}
  Let $u,w$ solve \eqref{eq:ch} and satisfy \eqref{eq:ch-error-assump}. Let $U_h, W_h$ solve \eqref{eq:ch-fem} with initial condition \eqref{eq:115}. We have that
  \begin{equation}
    \label{eq:3}
    \eps \sup_{t \in (0,T)} \norm{ u^{-\ell} - U_h }_{L^2(\Gamma_h(t))}^2
    + \int_0^T \norm{ w^{-\ell} - W_h }_{L^2(\Gamma_h(t))}^2
    \le C h^4,
  \end{equation}
  and
  \begin{equation}
    \label{eq:1}
    \eps \sup_{t \in (0,T)} \norm{ \nabla_{\Gamma_h} ( u^{-\ell} - U_h ) }_{L^2(\Gamma_h(t))}^2
    + \int_0^T \norm{ \nabla_{\Gamma_h} ( w^{-\ell} - W_h ) }_{L^2(\Gamma_h(t))}^2
    \le C h^2,
  \end{equation}
  with
  $C$ given by
  \begin{align*}
    C & =  c_\eps \sup_{t \in [0,T)} \norm{ u }_{H^2(\Gamma(t))}^2 + c_\eps \int_0^T \big(  \norm{ \md u }_{H^2(\Gamma(t))}^2 + \norm{ w }_{H^2(\Gamma(t))}^2 \big) \dd t.
  \end{align*} 
\end{theorem}

\subsection{Pointwise bound on the discrete solution}

In the following error analysis, a pointwise bound on the discrete solution uniformly in space and time will be extremely useful. This will allow us to convert the local Lipschitz property of $\psi$ and $\psi'$ into global results.

\begin{theorem}
  \label{thm:ptwise-bound}
  The discrete solution $U_h$ is bounded uniformly in space and time, independently of $h$, and we have the bound
  \begin{equation}
    \label{eq:107}
    \sup_{t \in (0,T)} \norm{ U_h }_{L^\infty(\Gamma_h(t))}^2 \le C_2(u_0).
  \end{equation}
\end{theorem}

\begin{proof}
  Let $F_h = W_h - \frac1\eps \psi'(U_h)$, then $F_h \in L^\infty(0,T; L^2(\Gamma_h(t)))$ with the estimate
  \begin{equation}
    \label{eq:114}
    \sup_{t \in (0,T)} \norm{ F_h }_{L^2(\Gamma_h(t))}^2 \le C_2(u_0).
  \end{equation}
  This follows immediately from \myref[Theorem]{thm:ch-fem-wp} and \myref[Lemma]{lem:ch-fem-bound-improved} combined with a Sobolev inequality (\myref[Lemma]{lem:sobolev-embedh}) and \eqref{eq:smoothu0-bound}.
  Furthermore, since $\phi_h = 1$ is an admissible test function in \eqref{eq:ch-fem-u}, the mean value of $F_h$ is zero:
  \begin{equation}
    \label{eq:62}
    \int_{\Gamma_h(t)} F_h \dd \sigma_h = 0.
  \end{equation}  
  
  We define $\tilde{F}_h = F_h^\ell / \mu_h^\ell$, so that
  \begin{equation*}
    \int_{\Gamma(t)} \tilde{F}_h \dd \sigma = \int_{\Gamma(t)} F_h^\ell \frac{1}{\mu_h^\ell} \dd \sigma = \int_{\Gamma_h(t)} F_h \dd \sigma_h = 0.
  \end{equation*}
  Let $\bar{u} \colon \G_T \to \R$ solve
  \begin{equation*}
    - \eps \Delta_\Gamma \bar{u} = \tilde{F}_h \quad \mbox{ on } \Gamma(t), \mbox{ and } \int_{\Gamma(t)} \bar{u} \dd \sigma = \int_{\Gamma_h(t)} U_h \dd \sigma_h \mbox{ for each } t \in (0,T).
  \end{equation*}
  Then it is clear that $\Pi_h \bar{u} = U_h$. Standard elliptic theory \cite{Aub82} and the $L^\infty$ bound on $\Pi_h$ \eqref{eq:Pih-infty} gives that
  \begin{align*}
    \norm{ U_h }_{L^\infty(\Gamma_h(t))} & = \norm{ \Pi_h \bar{u} }_{L^\infty(\Gamma_h(t))} \le c \norm{ \bar{u} }_{H^2(\Gamma(t))} \\ 
    & \le c \norm{ \tilde{F}_h }_{L^2(\Gamma(t))} \le c \norm{ F_h }_{L^2(\Gamma_h(t))}.
  \end{align*}
  We apply this inequality uniformly in time, with \eqref{eq:114}, to give the desired estimate. \qed
\end{proof}

\subsection{Splitting the error}

We split the error into two parts using the Ritz projection $\Pi_h$ from \myref[Section]{sec:ritz-projection}:
\begin{align*}
  u^{-\ell} - U_h & = ( u^{-\ell} - \Pi_h u ) + ( \Pi_h u - U_h ) = \rho^u + \theta^u \\
  w^{-\ell} - W_h & = ( w^{-\ell} - \Pi_h w ) + ( \Pi_h w - W_h ) = \rho^w + \theta^w.
\end{align*}
We note that from \myref[Theorem]{thm:Pih-bound}, we already have estimates for $\rho^u$ and $\rho^w$ and it is left to bound $\theta^u$ and $\theta^w$. Notice that, the assumptions in \eqref{eq:ch-error-assump} imply that $\theta^u \in S_h^T$ and  $\theta^w \in \tilde{S}_h^T$.

To derive equations for $\theta^u$ and $\theta^w$, we start by rewriting \eqref{eq:ch-fem-w} using the definition of $\Pi_h$ and \eqref{eq:ch-w} to obtain for $\phi_h \in S_h^T$ with lift $\varphi_h \in S_h^{\ell,T}$ that
\begin{equation}
  \label{eq:theta-u}
  \begin{aligned}
    & \dt m_h( \theta^u, \phi_h ) + a_h( \theta^w, \phi_h ) - m_h( \theta^u, \md_h \phi_h) \\
    & \quad = \big( m_h( \md_h \Pi_h u, \phi_h ) - m( \md_h u, \varphi_h ) \big)
    + \big( g_h( V_h; \Pi_h u, \phi_h ) - g( v_h; u, \varphi_h ) \big) \\
    & \quad\qquad + m( u, \md \varphi_h - \md_h \varphi_h ) \\
    & \quad =: E_1( \phi_h ) + E_2( \phi_h ) + E_3( \phi_h ).
  \end{aligned}
\end{equation}
Next, we rewrite \eqref{eq:ch-fem-u} using \eqref{eq:ch-u} this time to see for $\phi_h \in \tilde{S}_h^T$ with lift $\varphi_h \in \tilde{S}_h^{\ell,T}$ that
\begin{equation}
  \label{eq:theta-w}
  \begin{aligned}
    & \eps a_h( \theta^u, \phi_h ) + \frac1\eps m_h( \psi'( \Pi_h u) - \psi'(U_h), \phi_h )
    - m_h( \theta^w, \phi_h ) \\
    & \quad = \frac1\eps \big( m_h( \psi'(\Pi_h u), \phi_h ) - m( \psi'(u), \varphi_h ) \big)
    - \big( m_h (\Pi_h w, \phi_h ) - m( w, \varphi_h ) \big) \\
    & \quad =: E_4( \phi_h ) + E_5( \phi_h ).
  \end{aligned}
\end{equation}

The quantities $E_j(\phi_h)$, for $j=1,\ldots,5$, are consistency terms involving the approximation properties of the finite element spaces and the geometric perturbation.

\begin{lemma}
  \label{lem:E-bound}
  For $\phi_h \in S_h^T$ we have
  \begin{align}
    \label{eq:E1}
    \abs{ E_1( \phi_h ) } & \le c h^2 \big( \norm{ \md u }_{H^2(\Gamma(t))}
    + \norm{ u }_{H^2(\Gamma(t))} \big) \norm{ \phi_h }_{L^2(\Gamma_h(t))} \\
    \label{eq:E2}
    \abs{ E_2( \phi_h ) } & \le c h^2 \norm{ u }_{H^2(\Gamma(t))} \norm{ \phi_h }_{L^2(\Gamma_h(t))} \\
    \label{eq:E3}
    \abs{ E_3( \phi_h ) } & \le c h^2 \norm{ u }_{L^2(\Gamma(t))} \norm{ \phi_h }_{H^1(\Gamma_h(t))},
  \end{align}
  and for $\phi_h \in \tilde{S}_h^T$:
  \begin{align}
    \label{eq:E4}
    \abs{ E_4( \phi_h ) } & \le c \frac{h^2}{\eps} \norm{ u }_{H^2(\Gamma(t))}
    \norm{ \phi_h }_{L^2(\Gamma_h(t))} \\
    \label{eq:E5}
    \abs{ E_5( \phi_h ) } & \le c h^2 \norm{ w }_{H^2(\Gamma(t))} \norm{ \phi_h }_{L^2(\Gamma_h(t))}.
  \end{align}
\end{lemma}

\begin{proof}
  The proof is a combination of the geometric bounds from \myref[Section]{sec:geometric-estimates} and the bounds of $\Pi_h$ from \myref[Theorem]{thm:Pih-bound} and \myref[Lemma]{lem:md-Pih-bound}. \qed
\end{proof}

\subsection{Error bounds}

In this section, we derive bounds on $\theta^u$ and $\theta^w$ based on the error equations derived in the previous section and natural energy methods for the partial differential equation system go on to show the final error estimate.

To bound $\theta^u$ and $\theta^w$ we start by testing \eqref{eq:theta-u} with $\eps \theta^u$ and \eqref{eq:theta-w} with $\theta^w$ and subtract to see that
\begin{align*}
  & \eps \dt m_h( \theta^u, \theta^u ) + m_h( \theta^w, \theta^w ) \\
  & \quad = \eps m_h( \theta^u, \md_h \theta^u ) 
  + \frac1\eps m_h( \psi'(\Pi_h u) - \psi'(U_h), \theta^w ) \\
  & \quad\qquad + E_1( \eps \theta^u ) + E_2( \eps \theta^u ) +E_3( \eps \theta^u ) - E_4( \theta^w ) - E_5( \theta^w ).
\end{align*}
Applying \myref[Lemma]{lem:E-bound} and the transport lemma \eqref{eq:mh-transport}, with the local Lipschitz property of $\psi'$, this result gives that
\begin{equation}
  \begin{aligned}
    & \frac\eps2 \dt m_h( \theta^u, \theta^u ) + m_h( \theta^w, \theta^w ) \\
    & \quad \le c \frac\eps2 \norm{ \theta^u }_{L^2(\Gamma_h(t))}^2
    + \frac1\eps \norm{ \theta^u }_{L^2(\Gamma_h(t))} \norm{ \theta^w }_{L^2(\Gamma_h(t))} \\
    & \quad\qquad + c \eps h^2 \big( \norm{ \md u }_{H^2(\Gamma(t))} + \norm{ u }_{H^2(\Gamma(t))} \big) \norm{ \theta^u }_{L^2(\Gamma_h(t))} \\
    & \quad\qquad + \frac{c h^2}{\eps} \big( \norm{ u }_{H^2(\Gamma(t))} + \norm{ w }_{H^2(\Gamma(t))} \big) \norm{ \theta^w }_{L^2(\Gamma_h(t))} \\
    & \quad\qquad + c \eps h^2 \norm{ u }_{H^2(\Gamma(t))} \norm{ \nabla_{\Gamma_h} \theta^u }_{L^2(\Gamma_h(t))}.
  \end{aligned}
\end{equation}
We apply a Young's inequality to find that
\begin{equation}
  \label{eq:theta-uw}
  \begin{aligned}
    & \eps \dt \norm{ \theta^u }_{L^2(\Gamma_h(t))}^2
    + \norm{ \theta^w }_{L^2(\Gamma_h(t))}^2 \\
    & \quad \le \frac{1}{\eps^2} \norm{ \theta^u }_{L^2(\Gamma_h(t))}^2 + c \eps \norm{ \nabla_{\Gamma_h} \theta^u }_{L^2(\Gamma_h(t))}^2 \\
    & \quad\qquad + \frac{c h^4}{\eps^2} \big( \norm{ \md u }_{H^2(\Gamma(t))}^2 + \norm{ u }_{H^2(\Gamma(t))}^2 + \norm{ w }_{H^2(\Gamma(t))}^2 \big).
  \end{aligned}
\end{equation}

Next, in order to bound the $\nabla_{\Gamma_h} \theta^u$ term in the previous equation, we test \eqref{eq:theta-w} with $\theta^u$. Using \myref[Theorem]{thm:Pih-bound} and \myref[Lemma]{lem:E-bound} and the $L^\infty$ bound on $u$ and $U_h$, we have for some $\delta > 0$,
\begin{equation}
  \label{eq:grad-theta-u}
  \begin{aligned}
    \eps a_h( \theta^u, \theta^u ) 
    & = m( \theta^w, \theta^u ) - \frac1\eps m_h( \psi'( \Pi_h u) - \psi'(U_h), \theta^u )
    + E_4( \theta^u ) + E_5( \theta^u ) \\
    & \le c \frac1\eps \norm{ \theta^u }_{L^2(\Gamma_h(t))}^2 + \norm{ \theta^w }_{L^2(\Gamma_h(t))} \norm{ \theta^u }_{L^2(\Gamma_h(t))} \\
    & \qquad + c \frac{h^2}{\eps} \big( \norm{ u }_{H^2(\Gamma(t))}
    + \norm{ w }_{H^2(\Gamma(t))} \big) \norm{ \theta^u }_{L^2(\Gamma_h(t))} \\
    & \le c \frac1\eps \norm{ \theta^u }_{L^2(\Gamma_h(t))}^2 + \delta \norm{ \theta^w }_{L^2(\Gamma(t))}^2 \\
    & \qquad + c \frac{h^4}{\eps^2} \big( \norm{ u }_{H^2(\Gamma(t))}^2 + \norm{ w }_{H^2(\Gamma(t))}^2 \big).
    \end{aligned}
\end{equation}
Applying this bound in the right-hand side of \eqref{eq:theta-uw}, we may choose $\delta$ small enough so that
\begin{equation}
  \begin{aligned}
    & \eps \dt \norm{ \theta^u }_{L^2(\Gamma_h(t))}^2
    + \norm{ \theta^w }_{L^2(\Gamma_h(t))}^2 \\
    & \quad \le c \frac1\eps \norm{ \theta^u }_{L^2(\Gamma_h(t))}^2 
    + c \frac{h^4}{\eps^2} \big( \norm{ \md u }_{H^2(\Gamma(t))}^2
    + \norm{ u }_{H^2(\Gamma(t))}^2 + \norm{ w }_{H^2(\Gamma(t))}^2 \big).
  \end{aligned}
\end{equation}

We recall from \eqref{eq:115}: $U_{h,0} = \Pi_h u_0$, hence we have that $\theta^u|_{t=0} = \Pi_h u_0 - U_{h,0} = 0$. Applying a Gronwall inequality and integrating in time gives the following bounds on $\theta^u$ and $\theta^w$:
\begin{equation}
  \eps \sup_{t \in (0,T)} \norm{ \theta^u }_{L^2(\Gamma_h(t))}^2
  + \int_0^T \norm{ \theta^w }_{L^2(\Gamma_h(t))}^2 \dd t
  \le C  h^4,
\end{equation}
with $C = C( u, w, \eps, T )$ given by
\begin{equation*}
  C =  c_\eps \int_0^T \big(  \norm{ \md u }_{H^2(\Gamma(t))}^2
  + \norm{ u }_{H^2(\Gamma(t))}^2 + \norm{ w }_{H^2(\Gamma(t))}^2 \big) \dd t.
\end{equation*}

\begin{proof}[Proof of Theorem~\ref{thm:ch-error}]
  The previous bound can then be combined with the bounds on $\rho^u$ and $\rho^w$ from \myref[Theorem]{thm:Pih-bound} to give the $L^2$ error \eqref{eq:3}. One can also apply an inverse inequality to derive gradient bounds on $\theta^u$ and $\theta^w$ to give the $H^1$ error bound \eqref{eq:1}.\qed
\end{proof}

%--- numerical results ------------------------------------------------------%

\section{Numerical results}

The above finite element method discretised in time using semi-implicit time stepping. Given $U_0$ and a partition of time $0 = t_0, t_1, \ldots,t_M = T$, for $k=0,\ldots,M-1$, we find $(U_{k+1}, W_{k+1})$ as the solution the matrix system
\begin{align*}
  \M(t_{k+1}) U_{k+1} + ( t_{k+1} - t_k ) \S( t_{k+1} ) W_{k+1} & = \M(t_k) U_k \\
  \eps \S( t_{k+1} ) U_{k+1} - \M( t_{k+1} ) W_{k+1} & =  - \frac1\eps \Psi( U_{k} ).
\end{align*}
Full analysis of the fully discrete problem is left to future work. Based on ideas from \cite{DziEll12}, we expect stability subject to $\tau < \eps$ and convergence rate order $\tau + h^2$ for the discrete version of the norms in \myref[Theorem]{thm:ch-error}.

The method was implemented using the \textsf{ALBERTA} finite element toolbox \cite{SchSie05} and the full block linear system solved using a direct solver.

\subsection{Fourth-order linear problem}

We start by showing the derived orders of convergence can be achieved for a fourth order linear problem. We calculate with $\psi \equiv 0$ and choose $\eps = 0.1$. We couple $\tau \approx h^2$ to ensure we see the full order of convergence. The surface is given by $\Gamma(t) = \{ x \in \R^3 : \Phi(x,t)= 0\}$ with
\begin{equation}
  \label{eq:dziuk-sphere}
  \Phi(x,t) = \frac{x_1^2}{a(t)} + x_2^2 +x_3^2 - 1.
\end{equation}
We have chosen $a(t) = 1.0 + 0.25 \sin( 10 \pi t)$ and solve for $t \in (0,0.1)$. The exact solution is given by $u(x,t) = e^{-6t} x_1 x_2$, where right hand side $f$ is calculated from
\begin{equation*}
  f = u_t + v \cdot \nabla u + u \nabla_\Gamma \cdot v + \eps \Delta_\Gamma^2 u.
\end{equation*}
The convergence is shown in \myref[Table]{tab:ch-fourth-order} for the errors in the $L^2$ norm. The experimental order of convergence (eoc) is calculated via the formula \eqref{eq:eoc}: Given an error $E_i$ and $E_{i-1}$ at two different mesh sizes $h_i$ and $h_{i-1}$, we calculate the experimental order of convergence (eoc) by
\begin{equation}
  \label{eq:eoc}
  \mathrm{(eoc)}_i = \frac{ \log( E_i / E_{i-1} ) }{ \log( h_i / h_{i-1} ) }.
\end{equation}
The results for the $H^1$ norm are not shown here, however we observe first order convergence in $h$.

\begin{table}[tbh]
  \begin{center}

% u - u_h

  \begin{tabular}{ccc}
    \hline
    $h$ & $\norm{ u^{-\ell} - U_h }_{L^2(\Gamma_h(T))}$ & (eoc)  \\ 
    \hline
    $5.564983 \cdot 10^{-1}$ & $9.424750 \cdot 10^{-3}$ & ---  \\
    $2.866409 \cdot 10^{-1}$ & $3.001764 \cdot 10^{-3}$ & $1.724571$ \\
    $1.443332 \cdot 10^{-1}$ & $8.068147 \cdot 10^{-4}$ & $1.914955$ \\
    $7.229393 \cdot 10^{-2}$ & $2.033971 \cdot 10^{-4}$ & $1.993007$ \\
    \hline
  \end{tabular}

    \vspace{0.5cm}

% w - w_h

  \begin{tabular}{ccc}
    \hline
    $h$ & $\norm{ w^{-\ell} - W_h }_{L^2(\Gamma_h(T))}$ & (eoc) \\ 
    \hline
    $5.564983 \cdot 10^{-1}$ & $4.796888 \cdot 10^{-3}$ & ---  \\
    $2.866409 \cdot 10^{-1}$ & $1.432177 \cdot 10^{-3}$ & $1.821993$ \\
    $1.443332 \cdot 10^{-1}$ & $3.824468 \cdot 10^{-4}$ & $1.924429$ \\
    $7.229393 \cdot 10^{-2}$ & $9.651516 \cdot 10^{-5}$ & $1.991496$ \\
    \hline
  \end{tabular}

    \caption[Convergence for fourth-order linear problem]{Error table of the solution of a fourth-order linear problem with surface defined by \eqref{eq:dziuk-sphere}.}
    \label{tab:ch-fourth-order}
  \end{center}
\end{table}

\subsection{Cahn-Hilliard equation on a periodically evolving surface}

In this example, we consider the same surface as above but now with the full non-linearity as considered in the above analysis over the time interval $t \in (0,0.8)$.

The initial condition for the simulations was the interpolant of a small perturbation about zero given by
\begin{equation*}
  u_0 (x,y,z) = 0.1 \cos( 2 \pi x ) \cos( 2 \pi y ) \cos( 2 \pi z ).
\end{equation*}

We present two plots to show the behaviour of the numerical solution. First, in \myref[Figure]{fig:ch-conv}, we see that for short times we have good convergence of the solution. The second, \myref[Figure]{fig:ch-energy}, demonstrates that the energy does not decrease monotonically along solutions. Running for a longer time suggests that the solution converges to a time periodic solution. We show a plot of the solution at level 2 at different times in \myref[Figure]{fig:ch-dziuk}. The system is solved with a fixed time step of $10^{-4}$.

\begin{figure}[t]
  \centering
  \includegraphics{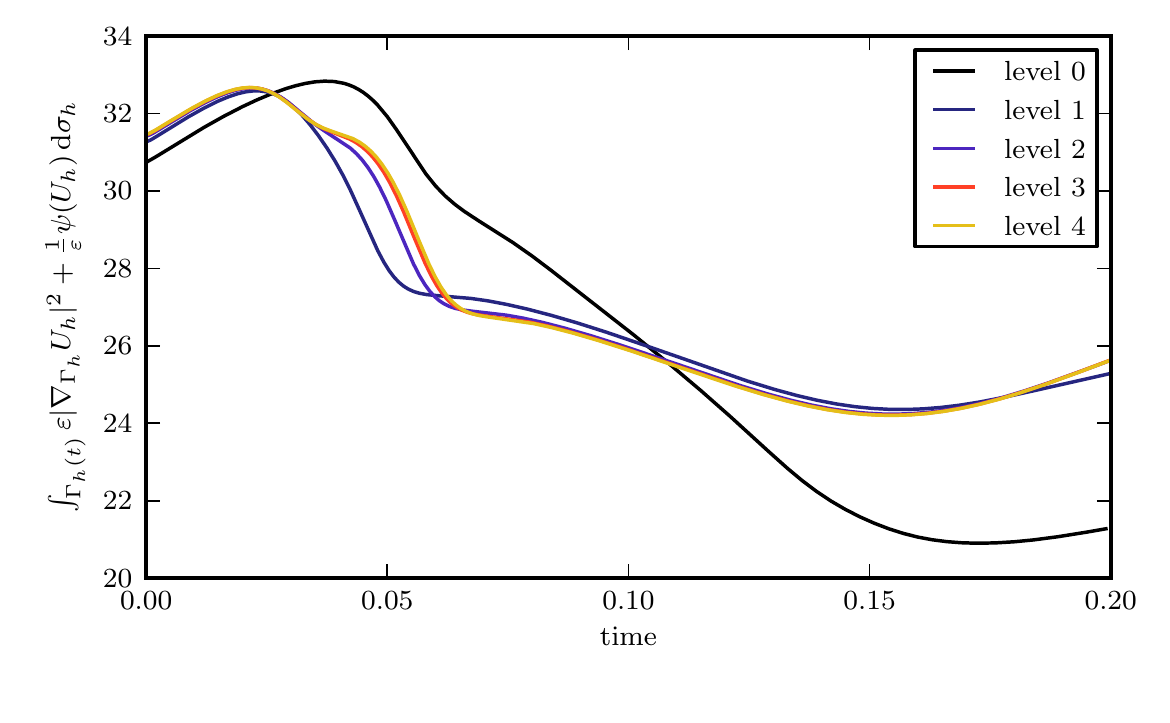}
  \caption[Convergence in energy for Cahn-Hilliard on an evolving surface]{A plot of the Ginzburg-Landau energy over five levels of refinement.}
  \label{fig:ch-conv}
\end{figure}

\begin{figure}[t]
  % column width: 338.0pt
  \centering
  \includegraphics{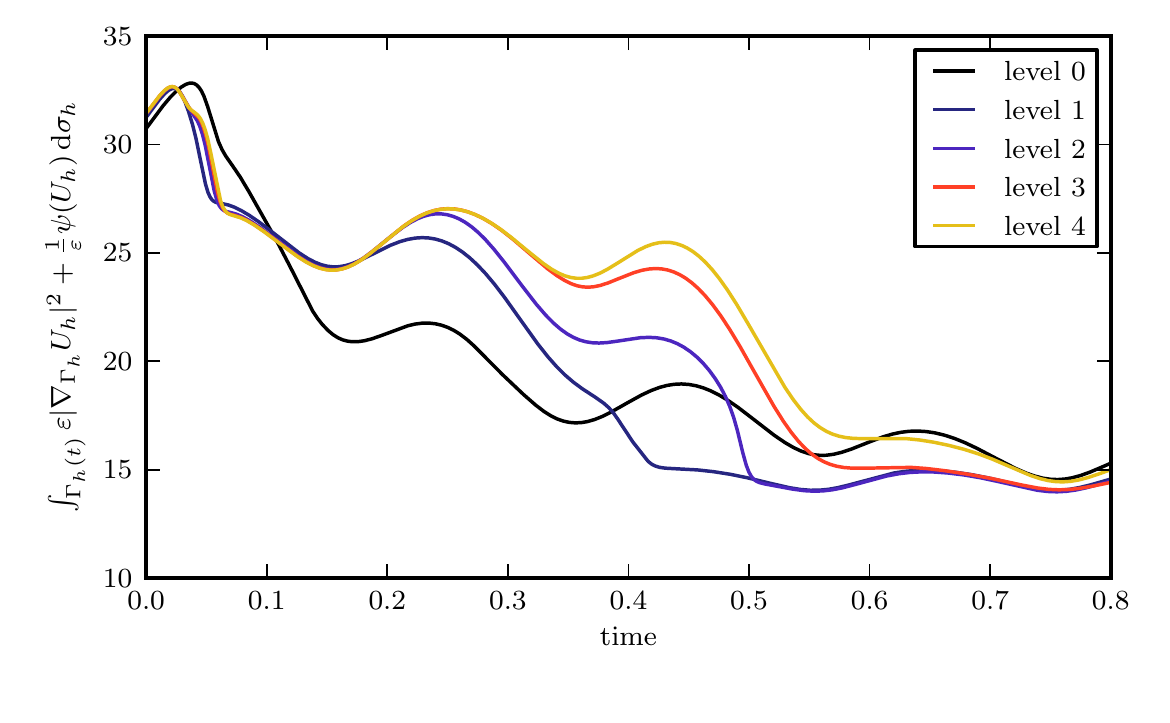}
  \caption[The decrease in energy for Cahn-Hilliard on an evolving surface]{A plot of the Ginzburg-Landau energy over five levels of refinement over a longer time interval.}
  \label{fig:ch-energy}
\end{figure}

\begin{figure}[p]
  \centering
  
  \begin{subfigure}[c]{0.3\textwidth}
    \centering
    \includegraphics[width=\textwidth]{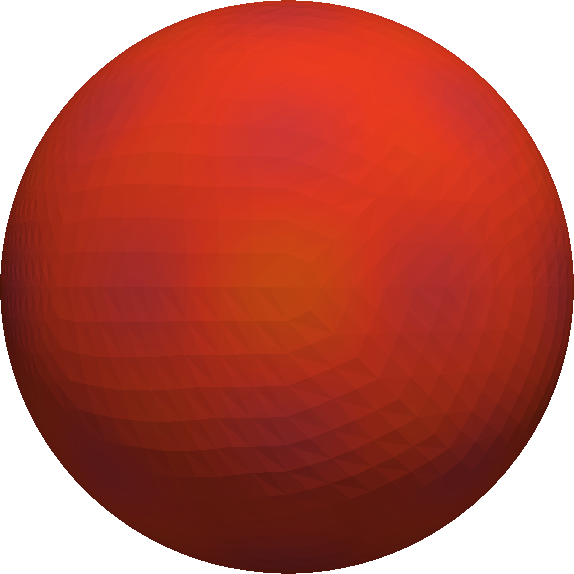}
  \end{subfigure}
  \quad
  \begin{subfigure}[c]{0.3\textwidth}
    \includegraphics[width=\textwidth]{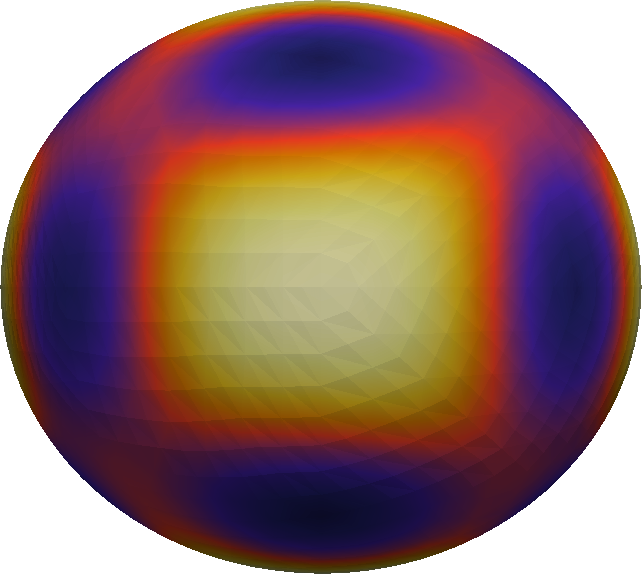}
  \end{subfigure}
  \quad
  \begin{subfigure}[c]{0.3\textwidth}
    \includegraphics[width=\textwidth]{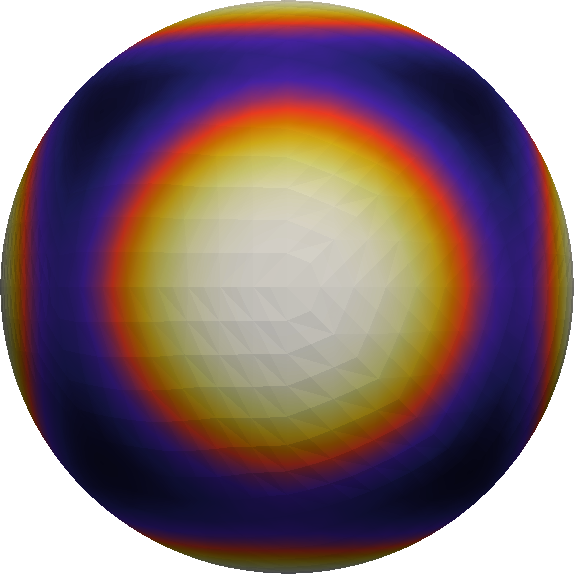}
  \end{subfigure}

  \vfill

  \begin{subfigure}[c]{0.3\textwidth}
    \includegraphics[width=\textwidth]{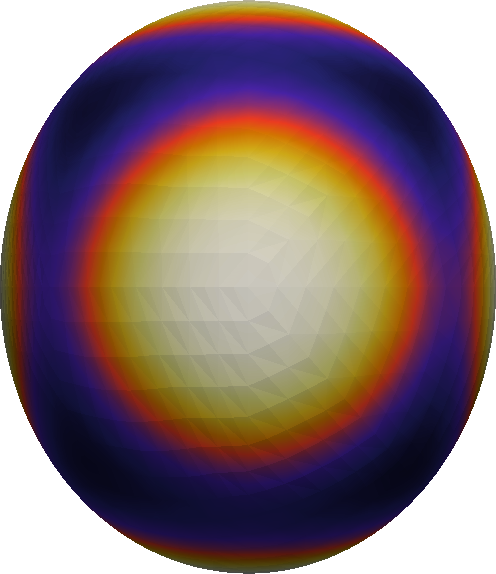}
  \end{subfigure}
  \quad
  \begin{subfigure}[c]{0.3\textwidth}
    \includegraphics[width=\textwidth]{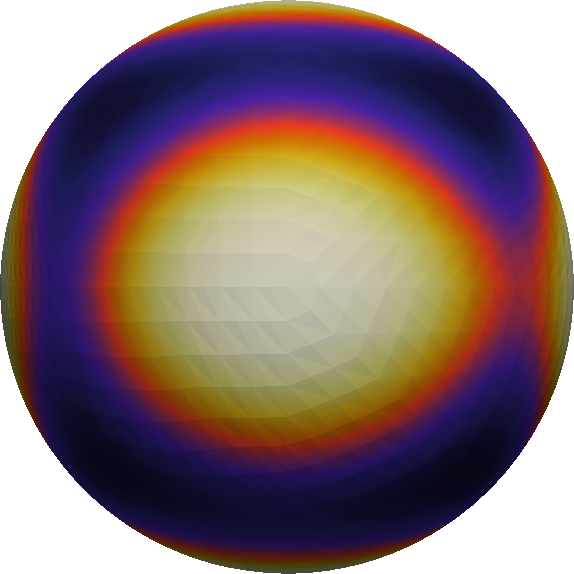}
  \end{subfigure}
  \quad
  \begin{subfigure}[c]{0.3\textwidth}
    \includegraphics[width=\textwidth]{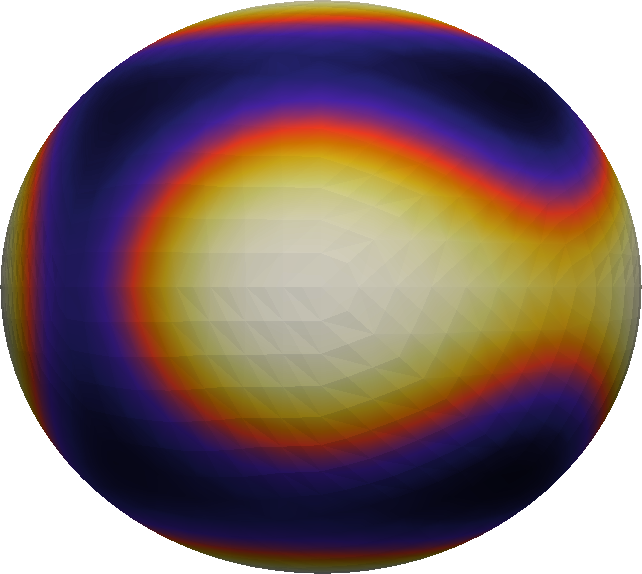}
  \end{subfigure}

  \vfill

  \begin{subfigure}[c]{0.3\textwidth}
    \includegraphics[width=\textwidth]{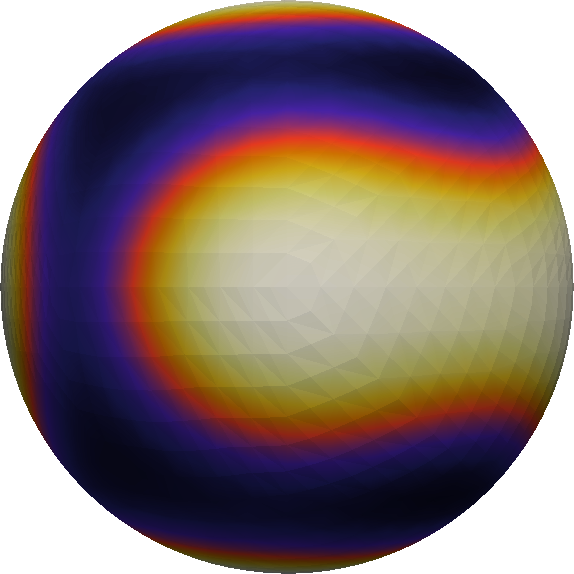}
  \end{subfigure}
  \quad
  \begin{subfigure}[c]{0.3\textwidth}  
    \includegraphics[width=\textwidth]{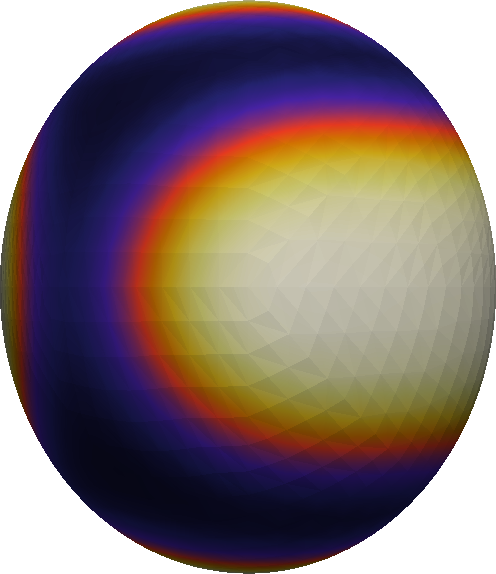}
  \end{subfigure}
  \quad
  \begin{subfigure}[c]{0.3\textwidth}
    \includegraphics[width=\textwidth]{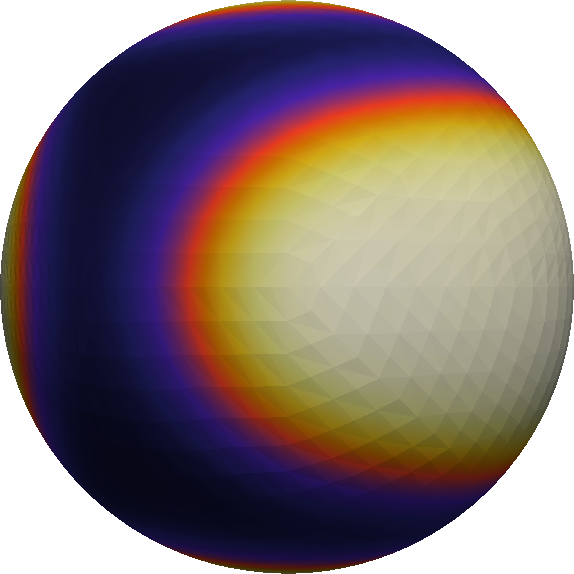}
  \end{subfigure}
  
  \vfill

  \begin{subfigure}[c]{0.3\textwidth}
    \includegraphics[width=\textwidth]{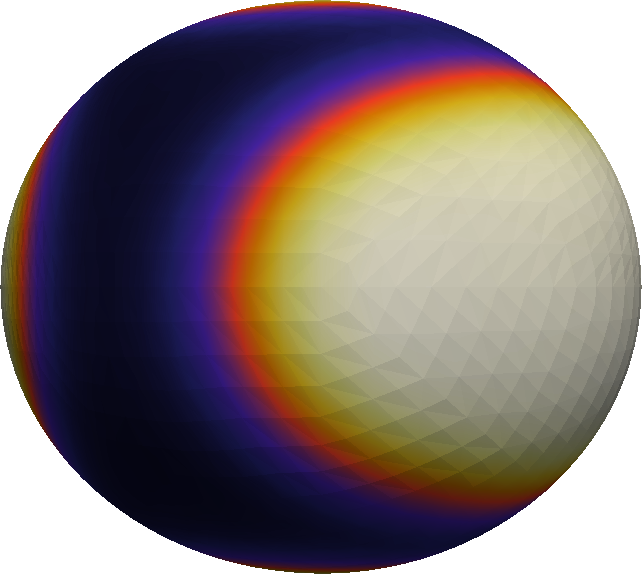}
  \end{subfigure}
  \quad
  \begin{subfigure}[c]{0.3\textwidth}
    \includegraphics[width=\textwidth]{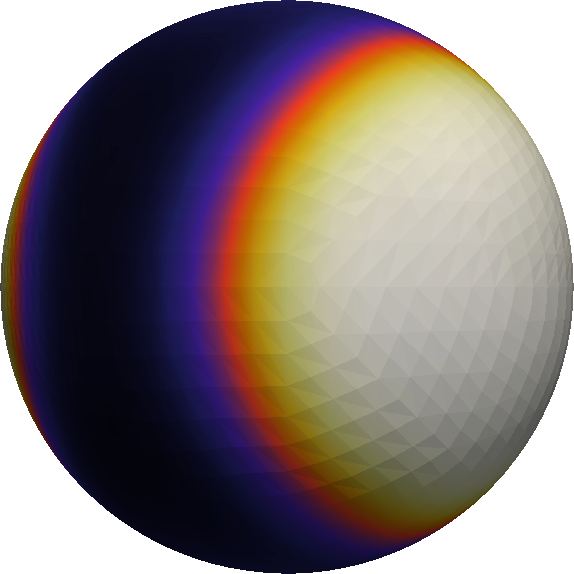}
  \end{subfigure}
  \quad
  \begin{subfigure}[c]{0.3\textwidth}
    \includegraphics[width=\textwidth]{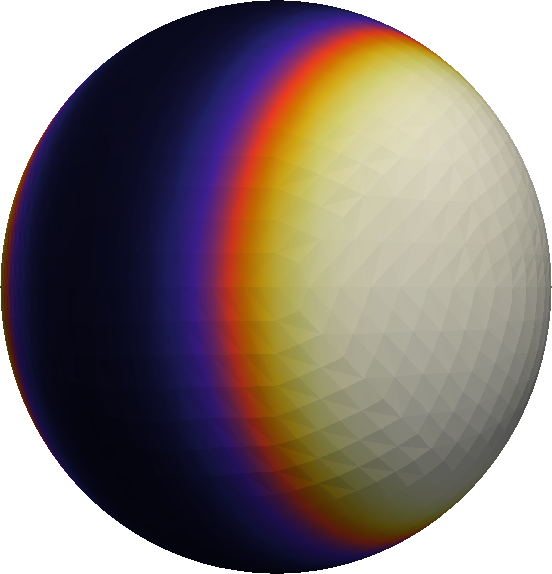}
  \end{subfigure}
  
  \caption[Plot of Cahn-Hilliard solution]{Plot of the solution of the Cahn-Hilliard equation at level two for time $t = 0.0, 0.05, 0.1, 0.15, 0.2, 0.25, 0.3, 0.35, 0.4, 0.45, 0.5, 1.0$. The colour scheme represents values between $-1$ and $1$.}
  \label{fig:ch-dziuk}
\end{figure}

\subsection{An example with tangential motion}

We show the flexibility of the method with an other example with larger surface deformation and tangential motion. The initial condition is taken to be a small random perturbation about zero. 

We take a surface given by the level set function
\begin{equation}
  \label{eq:ch-styles}
  \Phi(x,t) = x_1^2 + x_2^2 + a(t)^2 G( x_3^2 / L(t) ) - a(t)^2,
\end{equation}
where
\begin{align*}
  G(s) & = 200 s ( s - 199 / 100 ) \\
  a(t) & = 0.1 + 0.05 \sin( 2 \pi t ) \\
  L(t) & = 1 + 0.2 \sin( 4 \pi t ).
\end{align*}
In addition, we will prescribe a tangential velocity so that we will consider points moving according to
\begin{align*}
  X(t) = \left( X_1(0) \frac{a(t)}{a(0)},  X_2(0) \frac{a(t)}{a(0)},  X_3(0) \frac{L(t)}{L(0)} \right).
\end{align*}
We plot the solution at different times in \myref[Figure]{fig:ch-styles}. In particular, we notice that under this flow the nodes remain uniformly distributed.

\begin{figure}[p]
  \centering
  
  \begin{subfigure}[c]{0.30\textwidth}
    \includegraphics[width=\textwidth]{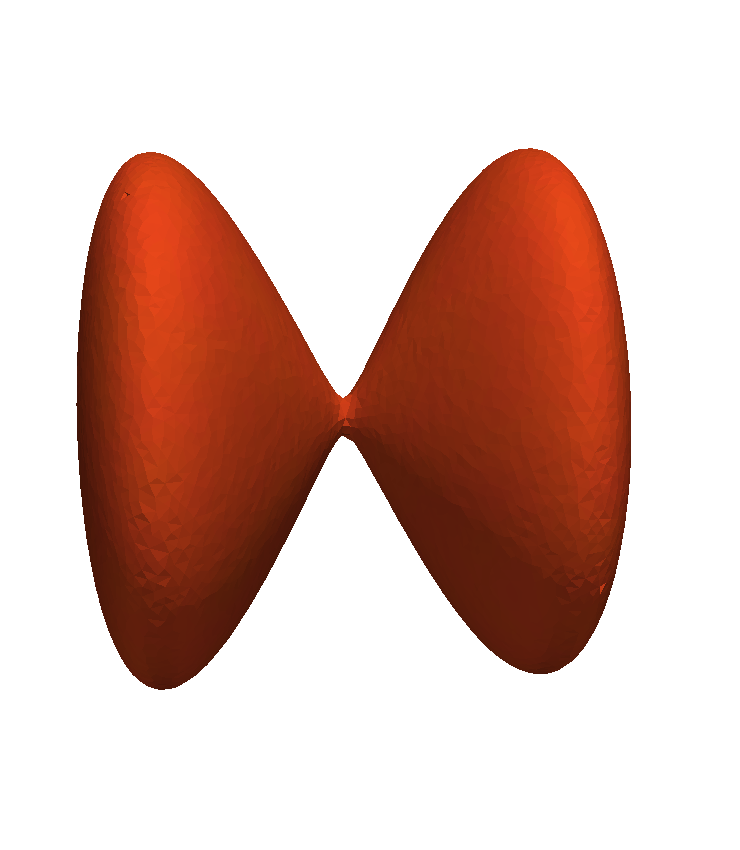}
  \end{subfigure}
  \qquad
  \begin{subfigure}[c]{0.30\textwidth}
    \includegraphics[width=\textwidth]{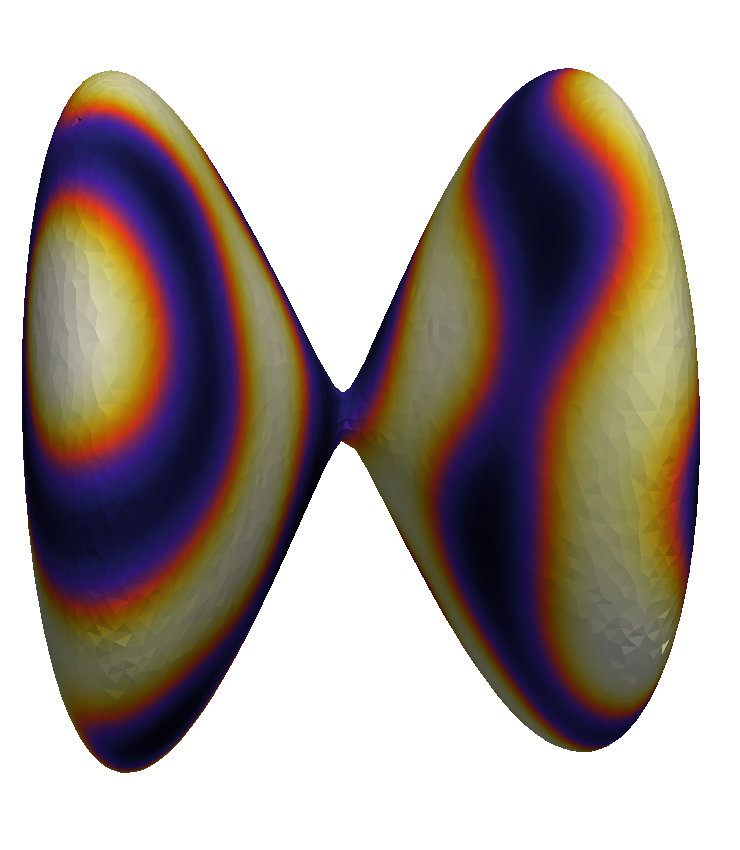}
  \end{subfigure}

  \begin{subfigure}[c]{0.30\textwidth}
    \includegraphics[width=\textwidth]{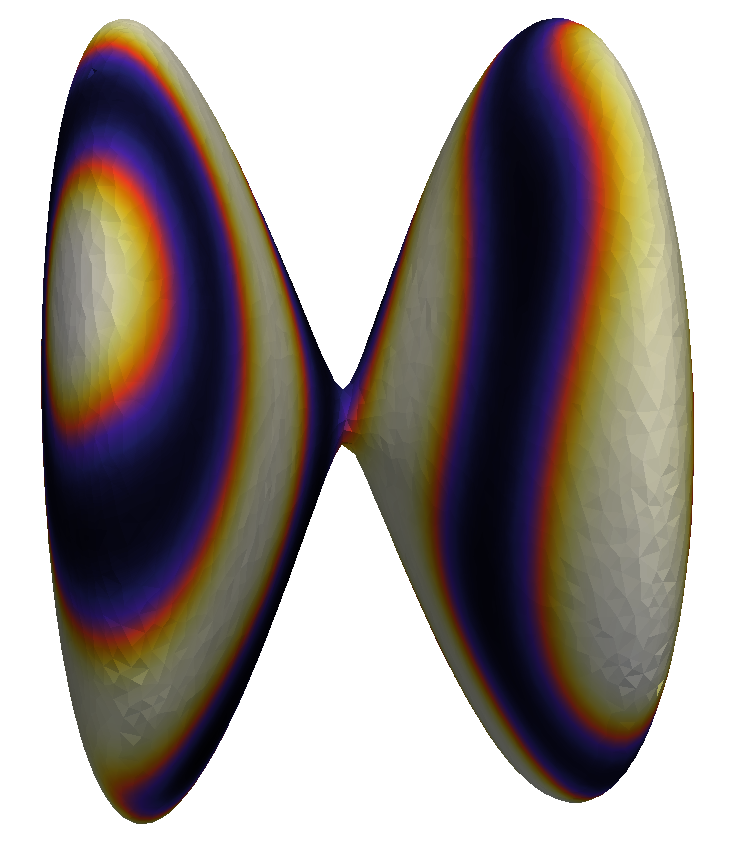}
  \end{subfigure}
  \qquad
  \begin{subfigure}[c]{0.30\textwidth}
    \includegraphics[width=\textwidth]{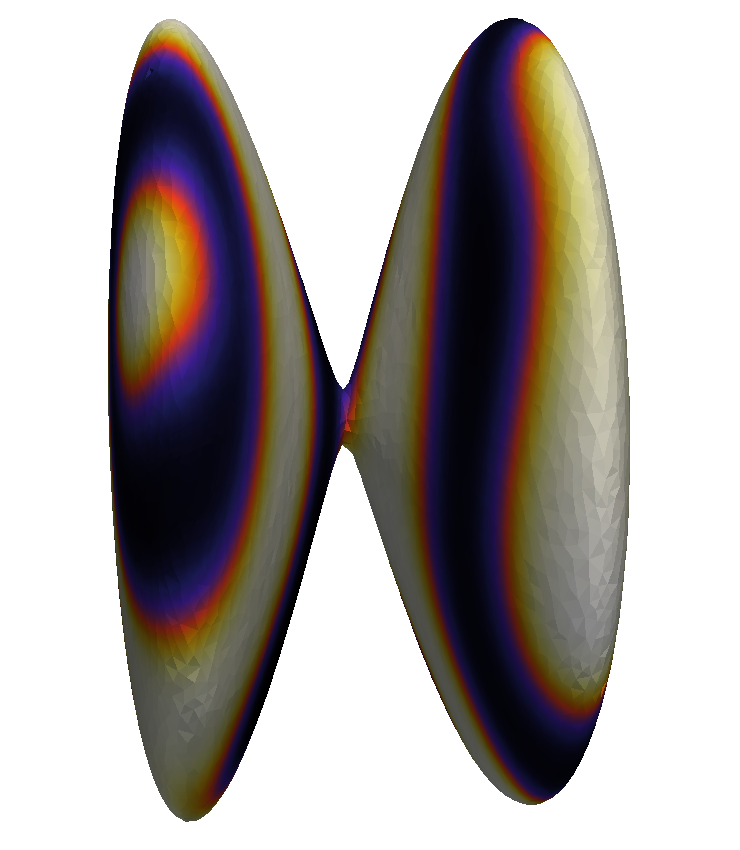}
  \end{subfigure}

  \begin{subfigure}[c]{0.30\textwidth}
    \includegraphics[width=\textwidth]{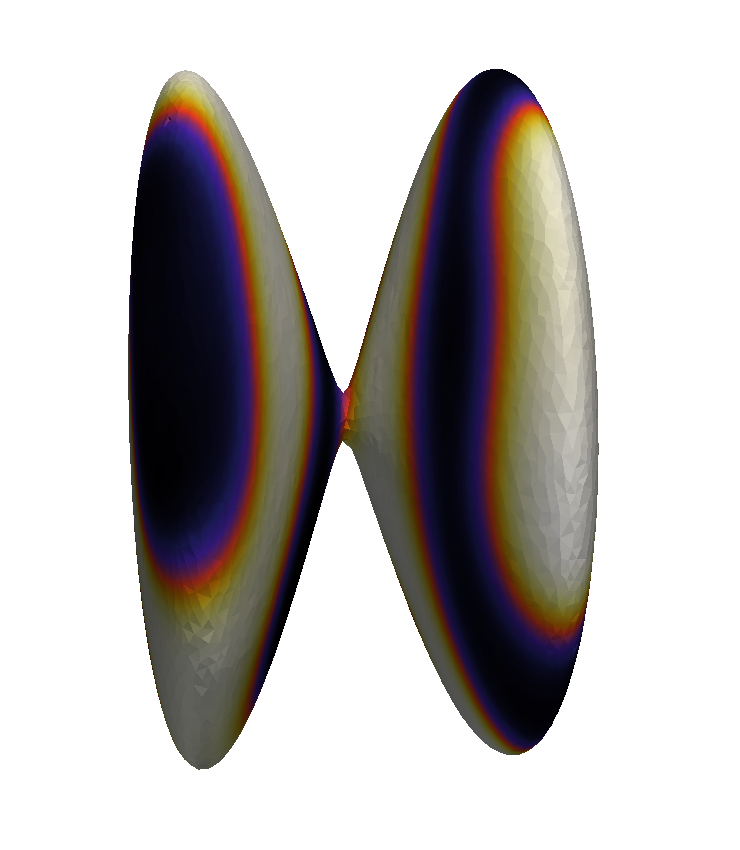}
  \end{subfigure}
  \qquad
  \begin{subfigure}[c]{0.30\textwidth}
    \includegraphics[width=\textwidth]{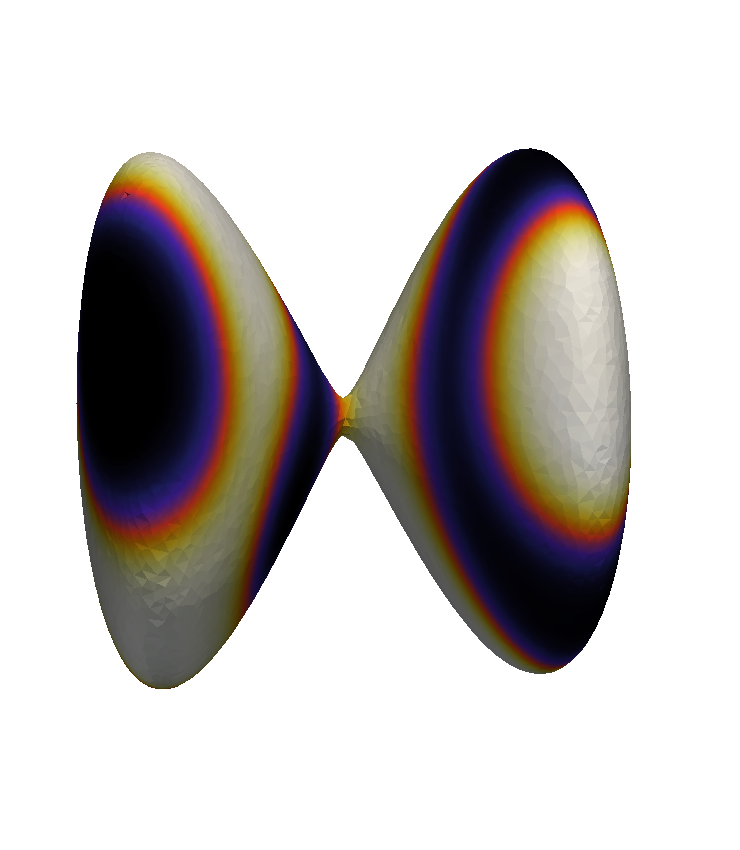}
  \end{subfigure}

  \begin{subfigure}[c]{0.30\textwidth}
    \includegraphics[width=\textwidth]{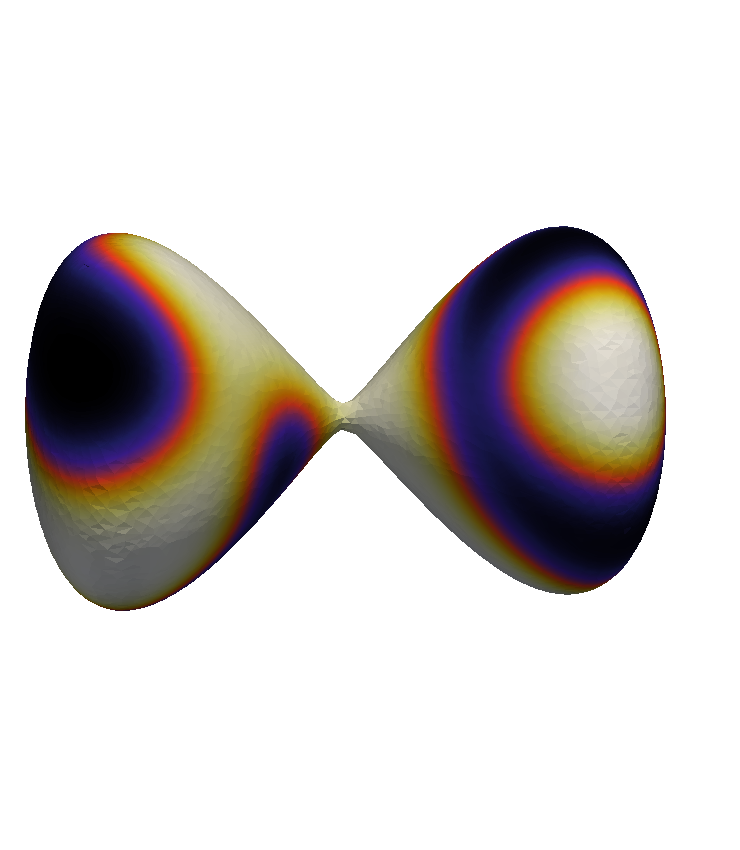}
  \end{subfigure}

  \caption[Example of Cahn-Hilliard on a evolving surface]{Plot of the solution on the surface defined by \eqref{eq:ch-styles} at times $t = 0, 0.1, 0.2, 0.3, 0.4, 0.5, 0.6$.}
  \label{fig:ch-styles}
\end{figure}

\begin{acknowledgements}
  The authors would like to thank Andrew Stuart and Endre S\"{u}lli for thoughtful comments and discussion which have improved this work greatly.
\end{acknowledgements}

\end{document}